\documentclass[10pt]{amsart}
\usepackage{amsmath}
\usepackage{amsfonts}
\usepackage{amssymb}
\usepackage{amsthm}
\usepackage{url}
\usepackage[all]{xy}
\usepackage{dsfont}
\usepackage{graphicx}
\usepackage{caption}
\usepackage{subcaption}
\usepackage{comment} 
\usepackage{stmaryrd}
\usepackage{hyperref}
\usepackage{todonotes}
\usepackage{color}
\usepackage{enumerate,tikz-cd}
\usepackage{fixltx2e}
\usepackage{MnSymbol}
\usepackage{comment}
\allowdisplaybreaks

\numberwithin{equation}{section}

\newtheorem{proposition}{Proposition}[section]

\newtheorem{lemma}[proposition]{Lemma}
\newtheorem{theorem}[proposition]{Theorem}
\newtheorem{corollary}[proposition]{Corollary}
\newtheorem{conjecture}{Conjecture}[section]

\theoremstyle{definition}
\newtheorem{remark}[proposition]{Remark}
\newtheorem{definition}[proposition]{Definition}
\newtheorem{example}[proposition]{Example}

\DeclareMathOperator{\GL}{GL}

\DeclareMathOperator{\Aut}{Aut}
\DeclareMathOperator{\Ric}{Ric}
\DeclareMathOperator{\Grass}{Grass}

\DeclareMathOperator{\SL}{SL}

\DeclareMathOperator{\Spec}{Spec}
\DeclareMathOperator{\des}{des}

\DeclareMathOperator{\DF}{DF}

\DeclareMathOperator{\Lie}{Lie}

\DeclareMathOperator{\ord}{ord}
\DeclareMathOperator{\mE}{E}
\DeclareMathOperator{\mM}{M}
\DeclareMathOperator{\mR}{R}
\DeclareMathOperator{\mH}{H}
\DeclareMathOperator{\Hilb}{Hilb}
\DeclareMathOperator{\dist}{dist}
\DeclareMathOperator{\moU}{U}

\newcommand{\R}{\mathbb{R}}
\newcommand{\C}{\mathbb{C}}

\newcommand{\Z}{\mathbb{Z}}

\newcommand{\Q}{\mathbb{Q}}

\newcommand{\pr}{\mathbb{P}}
\renewcommand{\epsilon}{\varepsilon}

\newcommand{\scO}{\mathcal{O}}
\newcommand{\scV}{\mathcal{V}}
\newcommand{\ddb}{i\partial \bar\partial}

\renewcommand{\L}{\mathcal{L}}
\newcommand{\X}{\mathcal{X}}

\renewcommand{\H}{\mathcal{H}}

\renewcommand{\phi}{\varphi}

\newcommand\FS{\mathrm{FS}}

\newcommand\mi{^{-1}}

\newcommand\na{{\mathrm{NA}}}

\newcommand\bbc{\mathbb{C}}

\newcommand\bbp{\mathbb{P}}

\newcommand\bbr{\mathbb{R}}

\newcommand\cC{\mathcal{C}}

\newcommand\cF{\mathcal{F}}

\newcommand\cH{\mathcal{H}}

\newcommand\cL{\mathcal{L}}

\newcommand\cO{\mathcal{O}}

\newcommand\cX{\mathcal{X}}

\newcommand{\norm}[1][\cdot]{\left\|#1\right\|}

\newcommand\triv{{\mathrm{triv}}}

\newcommand\NA{^{\mathrm{NA}}}

\newcommand\red{{\mathrm{red}}}

\newcommand\field{{\mathrm{K}}}
\newcommand\ring{{\mathrm{R}}}

\newcommand\Sym{\mathrm{Sym}}

\newcommand\cV{\mathcal{V}}

\title[Arcs, stability of pairs and the Mabuchi functional]{Arcs, stability of pairs and the Mabuchi functional}

\author[Ruadha\'i Dervan and R\'emi Reboulet]{Ruadha\'i Dervan and R\'emi Reboulet}

\address{Ruadha\'i Dervan, School of Mathematics and Statistics, University of Glasgow, University Place, Glasgow G12 8QQ, United Kingdom}\email{ruadhai.dervan@glasgow.ac.uk}
\address{R\'emi Reboulet, Department of Mathematical Sciences, Chalmers University of Technology, Chalmers tvärgata 3, Göteborg 412 96, Sweden}\email{rebouletremimath@gmail.com}
\begin{document}

\begin{abstract} We prove various results involving arcs---which generalise test configurations---within the theory of K-stability.

Our main result characterises coercivity of the Mabuchi functional on spaces of Fubini--Study metrics in terms of uniform K-polystability with respect to arcs, thereby proving a version of a conjecture of Tian. The main new tool is an arc version of a numerical criterion for Paul's theory of stability of pairs, for which we also provide a suitable generalisation applicable to pairs with nontrivial stabiliser.

We give two applications. Firstly, we give a new proof of a version of the Yau--Tian--Donaldson conjecture for Fano manifolds, along the lines originally envisaged by Tian---allowing us to reduce the general Yau--Tian--Donaldson conjecture to an analogue of the partial $C^0$-estimate. Secondly, for a (possibly singular) polarised variety which is uniformly K-polystable with respect to arcs, we show that the associated Cartan subgroup of its automorphism group is reductive. In particular, uniform K-stability with respect to arcs implies finiteness of the automorphism group. This generalises work of Blum--Xu for Fano varieties.\end{abstract}

\maketitle

\tableofcontents

\section{Introduction}

Let $(X,L)$ be a smooth polarised variety. The question of relating the existence of constant scalar curvature K\"ahler (cscK) metrics in $c_1(L)$ and algebro-geometric stability of $(X,L)$ has been a central goal of K\"ahler geometry for over thirty years. The guiding conjecture, due to Yau--Tian--Donaldson, is usually stated as follows:

\begin{conjecture}[Yau--Tian--Donaldson] The  class $c_1(L)$ admits a cscK metric if and only if $(X,L)$ is K-polystable.

\end{conjecture}

The notion of K-polystability involves test configurations, which are degenerations of $(X,L)$, embedded into a projective space, through $\C^*$-actions. To a test configuration $(\X,\L)$, one associates a numerical invariant called the Donaldson--Futaki invariant $\DF(\X,\L)$, and K-polystability asks for this invariant to be strictly positive unless the degeneration (or central fibre) $(\X_0,\L_0)$ is isomorphic to $(X,L)$.

The original statement of the Yau--Tian--Donaldson conjecture is expected to be false in general, as evidence suggests a condition stronger than K-polystability is needed for it to hold \cite{hattori-jstab, ACGTF}. One common strengthening is to instead consider uniform K-polystability \cite{uniform-twisted, bhj:duistermaat, hisamoto:toric}, where one asks for a uniform lower bound on the Donaldson--Futaki invariant with respect to a norm, and it is known \cite{bhj:asymptotics,hisamoto:toric} that the existence of a cscK metric implies uniform K-polystability. 

Another strengthening of K-polystability, due to Donaldson \cite{donaldson:bstab}, is to enlarge the class of test objects: he generalises test configurations to what are called \emph{arcs} or equivalently \emph{models} (see Proposition \ref{prop:arcsvmodels} for the precise relationship), which are degenerations over the formal disc $\Spec\C\llbracket z\rrbracket$, not required be induced by $\C^*$-actions, but where one fixes a trivialisation of the family away from the central fibre. The most common way of working with the Donaldson--Futaki invariant of a test configuration is through intersection theory, and the same definition can be made sense of for models, as we review in Section \ref{sec2}. We mention that the idea of associating numerical invariants to general classes of degenerations was also introduced by Wang \cite{xiaowei-height}, at around the same time.

We make progress towards the Yau--Tian--Donaldson conjecture by following this direction, combining models and uniform K-polystability. To state our main results, recall the Mabuchi functional on the space of K\"ahler potentials $\H$ relative to a fixed reference K\"ahler metric $\omega \in c_1(L)$, whose critical points are cscK metrics; we denote this functional $M: \H\to \R$. As we will be interested in the case where $(X,L)$ admits automorphisms, we fix further notation: we let $H\subset \Aut(X,L)$ be a maximal connected compact subgroup, and $T\subset H$ a maximal torus, and denote $\H^H$ the space of $H$-invariant potentials (where we assume $\omega$ is $H$-invariant). We similarly assume all arcs are $H^{\C}$-invariant, where $H^{\C}$ denotes the complexification. There is further a natural functional $J_T$ on $\H^H$, called the reduced $J$-functional, which plays the role of a norm, and for a subset $A\subset \H^H$ we will say that the Mabuchi functional is coercive on $A$ if there exists $\epsilon, \delta>0$ such that for all $\phi \in A$ we have $$M(\phi)\geq \epsilon J_T(\phi) - \delta.$$ We will be interested in two cases: when $A$ is the entirety of $\H^H$, and when $A$ consists of ($H$-invariant) Fubini--Study metrics under a fixed embedding of $X$ into projective space via the linear system $|rL|$; we denote the latter space by $\H^H_r$.

\begin{theorem}\label{intromainthm} The following are equivalent:
\begin{enumerate}[(i)]
\item $(X,L)$ is uniformly K-polystable with respect to arcs.
\item the Mabuchi functional is coercive on each space  $\H^H_r$ of Fubini--Study metrics.
\end{enumerate}
\end{theorem}

To be more precise, what we prove is an equivalence between the condition that there is an $\epsilon>0$ such that for all arcs $(\X,\L)$ we have $$\DF(\X,\L)\geq \epsilon \|(\X,\L)\|_{T^{\C}},$$ and the condition that there are $\epsilon, \delta_r$ strictly positive---where $\delta_r$ may depend on $r$---such that on $\H^H_r$ we have the coercivity condition $$M(\phi)\geq \epsilon J_T(\phi) - \delta_r.$$ This latter analytic condition is often called CM polystability, so what we prove is that CM polystability is equivalent to uniform K-polystability with respect to arcs.

We thus obtain a complete analytic characterisation of uniform K-polystability with respect to arcs, in terms of the Mabuchi functional. We also obtain more precise results, which hold for a fixed value of $k$ and consider only arcs induced by the embedding of $X$ into projective space via $|rL|$, provided we replace the Donaldson--Futaki invariant with the non-Archimedean Mabuchi functional (which for similar reasons as above makes sense not only for test configurations but also for arcs). We note that by Berman--Darvas--Lu \cite{BDL:aens}, the existence of a cscK metric implies coercivity of the Mabuchi functional on $\H^H$---hence on each $\H^H_r$---producing the following corollary.
 
 \begin{corollary}\label{coro:intro1} If the class $c_1(L)$ admits a cscK metric, then $(X,L)$ is uniformly K-polystable with respect to arcs.
 \end{corollary}
 
 In different directions, this strengthens work of Donaldson \cite{donaldson:algebraic} and Sz\'ekelyhidi \cite{sze:filtr} (who proved existence implies K-stability with respect to arcs under a discrete automorphism group assumption) and Boucksom--Hisamoto--Jonsson \cite{bhj:asymptotics} and Hisamoto \cite{hisamoto:toric} (who proved that existence implies uniform K-polystability with respect to test configurations, see also Li \cite{chili:ytd}). We are primarily interested in the other implication of Theorem \ref{intromainthm}, which is completely new. 
 
 We mention that a version of Theorem \ref{intromainthm} has been sought in the field for a long time using test configurations, first conjectured by Tian \cite{tian:kems}, and although the numerous works aiming to prove it all contain (unfixable) errors\footnote{See Paul--Sun--Zhang for a discussion of this prior work \cite[Remark 1.4]{paulsunzhang}.}, they have been influential to our work.

\subsection*{Tian's programme} To put Theorem \ref{intromainthm} into  context, we recall Tian's three-step programme to prove the Yau-Tian-Donaldson conjecture for Fano manifolds \cite{tian:kems}, which is expected to also apply in general to the cscK problem. 

The first step is completely analytic, and requires showing that the existence of a cscK metric is equivalent to coercivity of the Mabuchi functional, namely the existence of $\epsilon, \delta>0$ such that on $\H^H$ $$M(\phi)\geq \epsilon J_T(\phi) - \delta.$$ This first step is now a deep theorem of Chen--Cheng \cite{chencheng:ii,chencheng:iii}, with a version of this having been originally proven by Tian \cite{tian:kems} in the setting of K\"ahler--Einstein metrics on Fano manifolds  and with Darvas--Rubinstein having proven the version of the result stated here \cite{dar:rubinstein}. 

The second step is to show that coercivity of the Mabuchi functional on the full space $\H^H$ is equivalent to coercivity on each space  $\H^H_r$ of Fubini--Study metrics (where again, the $\delta_r$ involved may depend on $k$). This step is known as Tian's ``partial $C^0$-estimate'' and has been established by Sz\'ekelyhidi \cite{sze:partial} for Fano manifolds (this step is often stated as a condition on Bergman kernels, but arguments of Tian explain the relationship with this consequence for the Mabuchi functional, as we recall in the last  section of our work). 

The third step is where the algebro-geometric stability condition enters, and requires showing that K-polystability is equivalent to coercivity of the  Mabuchi functional on each space  $\H^H_r$ of Fubini--Study metrics. This step is open even in the Fano case (that is, without appealing to the resolution of the Yau--Tian--Donaldson conjecture for Fano manifolds), and in general it is expected that---for this step to hold---one needs to strengthen the notion of K-polystability. This third step is the one that Theorem \ref{intromainthm} gives a solution to, through arcs.
 
As all steps of Tian's programme are now known in the Fano case, we obtain the following consequence of Theorem \ref{intromainthm}, proving the arc version of the Yau--Tian--Donaldson conjecture along the lines originally envisaged by Tian \cite{tian:kems}:

\begin{corollary}\label{coro:intro2} A Fano manifold admits a K\"ahler--Einstein metric if and only if it is uniformly K-polystable with respect to arcs.

\end{corollary}

The stronger statement that K-polystability is equivalent to the existence of a K\"ahler--Einstein metric on Fano manifolds is due to Chen--Donaldson--Sun \cite{ytd:cds}, Berman \cite{ber:kpoly} and Tian \cite{tian:kems}, and there have since been several other proofs, often involving stronger variants of K-polystability \cite{bbj, KZ-quantization,DK-twisted,datszek,csw:kahlerricci}. In the general polarised case, Theorem \ref{intromainthm} along with the aforementioned work of Chen--Cheng  \cite{chencheng:iii,chencheng:ii} means that the first and third steps of Tian's programme are completed, leaving the second step as the outstanding one. 

Separately, Li and Berman--Boucksom--Jonsson \cite{chili:ytd, bbj} have reduced the uniform Yau--Tian--Donaldson conjecture (involving uniform K-polystability with respect to test configurations) to the non-Archimedean entropy approximation conjecture. It follows from their work and Theorem \ref{intromainthm} that the non-Archimedean entropy approximation conjecture would imply the version of the partial $C^0$-estimate explained as the second step of Tian's programme here, since uniform K-polystability with respect to arcs implies uniform K-polystability with respect to test configurations.

\subsection*{Stability of pairs} The main tool we use in our work is Paul's stability of pairs. For a fixed embedding of $X$ into projective space via $|rL|$, classically we obtain a Chow point $R_X$ lying inside the Chow variety. We furthermore obtain a second point $\Delta_X$ constructed in a related manner, and by choosing a lift we may consider these as lying in two separate vector spaces $V$ and $W$, which by construction admit actions of $G=\GL(H^0(X,rL))$. The group $G$ can be viewed as moving $X$ inside projective space, or dually as moving the Fubini--Study metric on $X$, and so we may consider the Mabuchi functional as inducing a functional $M: G \to \R$. A remarkable result of Paul \cite{paul:hyperdiscriminants}, which in essence states that the Mabuchi functional is an algebraic object, shows that there exist Hermitian inner products on $V$ and $W$ such that $$M(g) =\log|R_{g(X)}| - \log|\Delta_{g(X)}|.$$ A similar result holds for $J_T$. 

Motivated by these results, Paul introduces the following general setup. Consider a pair of vectors $(v,w) \in V\oplus W$ where $V$ and $W$ are vector spaces with actions of a (complex) linear algebraic group $G$. Paul then defines the pair $(v,w)$ to be \emph{semistable} if $$\overline{G.[v,w]} \cap \pr(0\oplus W) = \varnothing.$$ Endowing $V$ and $W$ with Hermitian inner products, Paul then proves that the pair $(v,w)$ is semistable if and only if there exists a $\delta>0$ such that for all $g\in G$ $$\log|g.v| - \log|g.w| \geq -\delta;$$ we call this condition \emph{analytic semistability}. It follows that boundedness of the Mabuchi functional on this space of Fubini--Study metrics can be characterised in terms of semistability of the pair $[v,w]$.

We begin by generalising Paul's theory to pairs admitting automorphisms, leading to the notion of a \emph{polystable} pair. Loosely, for  a maximal torus $T^{\C}$ lying in the automorphism group of the pair $[v,w]$, we define polystability of the pair by considering the geometry of a point in an associated Hilbert scheme whose underlying variety is defined by taking the closure of an associated $T^{\C}$-orbit. We also introduce is a natural version of coercivity which incorporates the $T^{\C}$-action, and hence a notion of \emph{analytic polystability} of $[v,w]$. We finally define \emph{numerical polystability} through arcs; in this context, an arc is  simply a morphism $\bbc\llparenthesis z\rrparenthesis \to G$, or equivalently an element of $G(\bbc\llparenthesis z\rrparenthesis)$. We associate a numerical invariant $\nu(\rho,[v,w])$ to $\rho$, extending the definition given by Wang and Donaldson for a single vector \cite{xiaowei-height, donaldson:bstab}, and use this to define numerical polystability by again incorporating the $T^{\C}$-action. In the polystable setting (but not the semistable or stable settings), as a hypothesis we assume a norm functional is proper, which is automatic in our applications. Our main result in this context is the following:

\begin{theorem}\label{arcHM} The following are equivalent:
\begin{enumerate}[(i)]
\item the pair $[v,w]$ is polystable;
\item the pair $[v,w]$ is analytically polystable;
\item the pair $[v,w]$ is numerically polystable.
\end{enumerate}
\end{theorem}

This gives a complete numerical  understanding of polystability of pairs through arcs, as well as through the behaviour of associated log norm functionals. We also prove versions of this result for semistability and stability, in the sense of Paul. The equivalence of $(i)$ and $(ii)$ is due to Paul in the setting of semistability, while the equivalence with $(iii)$ is new  even in the semistable case. The equivalence of numerical semistability and semistability is relatively straightforward, with the polystable case---which is the most technical part of our work---being  more challenging.

A semistable version of Theorem \ref{arcHM}---giving a numerical characterisation of semistability of pairs---has been desired in the field for many years. Previous strategies towards proving Theorem \ref{intromainthm} have aimed to prove a Hilbert--Mumford criterion for semistability of pairs, characterising semistability through weights associated to one-parameter subgroups of $G$. In light of recent examples due to Paul--Sun--Zhang \cite{paulsunzhang} (which inspired our work), this is impossible in general: semistability of pairs cannot be characterised through one-parameter subgroups\footnote{Their examples imply that all prior work towards Theorem \ref{intromainthm} using test configurations is erroneous.}. Thus Theorem \ref{arcHM} provides an optimal numerical criterion for polystability of pairs. 

The usage of arcs in the classical Hilbert--Mumford theorem in geometric invariant (GIT) goes back to Mumford's original work  \cite[Section 2.1]{book:git} and was highlighted by Donaldson \cite[Section 2.1]{donaldson:bstab}. Theorem \ref{arcHM} can thus further be thought of as an optimal replacement for the classical Hilbert--Mumford criterion that applies much more generally, both for differences of ample line bundles (as is the case for our applications) and for arbitrary (complex) linear algebraic groups, that may not be reductive. 

Theorem \ref{intromainthm} follows from a combination of Paul's work and Theorem \ref{arcHM}, once we identify the weight of an arc $(\X,\L)$ with the weight of the associated arc in $G$. To do this, we argue indirectly and identify the non-Archimedean Mabuchi functional as the quantity governing the asymptotics of the Mabuchi functional along the arc, building on work of Phong--Ross--Sturm \cite{phongrosssturm} and Boucksom--Hisamoto--Jonsson \cite{bhj:asymptotics}. 

The point of the Hilbert--Mumford criterion in GIT is to render stability (in the sense of GIT) practical and effective: it can be checked in explicit examples. In this way, our work similarly makes Paul's theory of stability of pairs practical and effective: most results on K-stability of general polarised varieties use the intersection-theoretic formulation of Wang and Odaka \cite{xiaowei-height, odaka2}, which is the one that we employ also for arcs. For example, Odaka's technique applies almost \textit{verbatim} to show that Calabi--Yau varieties with log terminal singularities are uniformly K-stable with respect to arcs \cite{odaka-2}. Hence by a variant of Theorem \ref{arcHM} proven here, the associated pair is stable in the sense of Paul for any embedding in projective space when the variety is in addition smooth (see Paul \cite[Corollary 1.2]{paul:hyperdiscriminants} for a more analytic proof of this statement). 

\subsection*{Reductivity} We end with one further application of our techniques. Paul has proven \cite[Proposition 4.10]{paul-stablepairs-13} that, in general, a stable pair has finite automorphism group. Thus it follows from Theorem \ref{intromainthm} and Paul's work that smooth polarised varieties that are uniformly K-stable with respect to arcs have finite automorphism group. We give a proof of this statement that uses arcs more directly, applies further to singular varieties, and generalises to give reductivity statements instead assuming uniform K-polystability with respect to arcs. 

To state our result, denote by $\Aut_{0}(X,L)^{T^{\C}}$ the centraliser of the maximal torus inside the connected component $\Aut_0(X,L)$ of the identity in the automorphism group $\Aut(X,L)$ of $(X,L)$, with our definition of uniform K-polystability implying  all arcs  considered are $T^{\C}$-invariant. 

\begin{theorem}\label{intro:automs}
If $(X,L)$ is uniformly K-polystable with respect to arcs, then the group $\Aut_0(X,L)^{T^{\C}}$ is reductive. In particular, if $(X,L)$ is uniformly K-stable with respect to arcs, then its automorphism group is finite.
\end{theorem}

We emphasise here that we allow $(X,L)$ to be singular. The group $\Aut_{0}(X,L)^{T^{\C}}$ is often called the Cartan subgroup associated to the maximal torus, and the result further implies that $\Aut_{0}(X,L)^{T^{\C}}$ is actually isomorphic to $T^{\C}$ itself. In practice, this means we can---for example---rule out the automorphism group taking the form $(\C,+)$ or $(\C,+) \oplus \C^*$, but not $\C^*\ltimes (\C,+)$.  In the  Fano setting, the second part of Theorem \ref{intro:automs} is due to Blum--Xu \cite{kmod:separatedness}, who prove the stronger statement that K-stable Fano varieties have finite automorphisms (though we do not expect that this result holds for general polarised varieties for similar reasons as to why the Yau--Tian--Donaldson conjecture must employ a stronger notion of K-stability). Alper--Blum--Halpern-Leistner--Xu \cite{kmod:reductivity} have proven K-polystability of a Fano variety implies reductivity of its automorphism group, which is in the same direction as, but stronger than, Theorem \ref{intro:automs}.
 
Theorem \ref{intro:automs} is one of extremely few general results around any version of K-stability in the non-Fano setting. The constraint on the automorphism group provided by Theorem \ref{intro:automs} is further one of a handful of results needed in the construction of moduli of polarised varieties through uniform K-polystability with respect to arcs, and gives perhaps only the second result in this direction after Odaka's work \cite{odaka} on the singularities of K-semistable varieties (with respect to test configurations). We remark that Donaldson's proof \cite{donaldson:algebraic} of Zariski openness of the K\"ahler--Einstein condition assuming finiteness of automorphisms employs arcs in a crucial way (a result also proven differently in \cite{odaka:moduli}, again using test configurations). We therefore hope that the perspective of arcs, which we take here, will prove useful in other aspects of the construction of moduli of varieties.

\subsection*{Acknowledgements.} We thank Hamid Abban, S\'ebastien Boucksom, Siarhei Finski, Yoshinori Hashimoto, Sean Paul, Jacob Sturm and Xiaowei Wang for discussions on this and related topics. The authors are grateful to the Isaac Newton Institute for Mathematical Sciences, Cambridge, for support during the programme ``New equivariant methods in algebraic and differential geometry'' , supported by EPSRC grant number EP/R014604/1, where our work began. RD was funded by a Royal Society University Research Fellowship (URF\textbackslash R1\textbackslash 201041), and  RR was funded by a Knut and Alice Wallenberg Foundation grant.

\subsection*{Notation} Throughout, we fix a smooth polarised variety $(X,L)$; we will occasionally relax the smoothness condition and will explicit state this when doing so. We set $\field:=\bbc\llparenthesis z\rrparenthesis$ the field of complex Laurent series and $\ring:=\bbc\llbracket z\rrbracket$ its valuation ring. For $X$ a scheme defined over $S$ and for $S'\to S$ a morphism, we denote by $X_{S'}$ the basechange of $X$ to $S'$. Thus  we denote by $(X_\field,L_\field)$ the base change of $(X,L)$ to $\field$, and denote by $X(\field)$ the set of $\field$-rational points. Likewise, if $V$ is a vector space over $\bbc$, we denote by $V_\field$ its ground field extension $V\otimes_\bbc \field$ to $\field$.

\section{K-stability and models}\label{sec2}

Here we define uniform K-polystability with respect to arcs, by first reviewing the theory of models in Section \ref{subsect:21}, before developing associated numerical invariants in Section \ref{subsect:22} through intersection theory.

\subsection{Models and automorphisms}\label{subsect:21}

\subsubsection{Generalities on models} 

Let $(X,L)$ be an $n$-dimensional normal polarised variety, so that by definition $L$ is an ample $\Q$-line bundle. A main role in the present work will be played by \emph{models}, which are a class of degenerations of $(X,L)$, generalising the test configurations involved in the traditional approach to K-stability.

\begin{definition}A \textit{model} of $(X,L)$ is the data of:
\begin{enumerate}[(i)]
	\item a scheme $\X$ with a flat, projective morphism $\pi: (\X,\L) \to \Spec \ring$;
    \item a relatively ample $\Q$-line bundle $\L$ on $\cX$;
	\item an isomorphism $(\cX_\field, \cL_\field) \simeq (X_\field, L_\field);$
\end{enumerate}
We call a model  \emph{normal} if $\X$ is normal.  The \textit{central fibre} of the model is the base change $(\cX_0,\cL_0)$.  If $r\L$ is relatively very ample, we further say that $(\X,r\L)$ is of \emph{exponent} $r$. 

\end{definition}

We will often omit the implicit data of the morphism $\pi$ and the isomorphisms in our notation, simply stating that $(\cX,\cL)$ is a model of $(X,L)$. We will also allow $\L$ to be relatively semiample in the above definition; this will be clear from context.

\begin{example}\label{ex:different-models}

Let $(\X_C,\L_C) \to C$ be a flat family over a smooth curve, or over the unit disc, and fix a point $0\in C$ such that the fibres away from $0$ are isomorphic to $(X,L)$. Then, as $0\in C$ is smooth, the local ring of $C$ at $0$ is isomorphic to the ring of power series $\C\llbracket t\rrbracket = \ring$. The restriction $(\X_{\ring},\L_{\ring}) = (\X,\L)\times_C \Spec \ring$ to the formal neighbourhood of $0$ in $C$ then defines a model for $(X,L)$.

Conversely, suppose $(\X,\L)$ is a model. In a formal neighbourhood of the origin, the trivial family $(X_{\C},L_{\C})$ is isomorphic to $(X_{\ring},L_{\ring})$, hence $(X_{\C^*},L_{\C^*})$ contains a formal neighbourhood isomorphic to $(X_{\field},L_{\field})$. Through the isomorphism $(\X_{\field},\L_{\field}) \cong (X_{\field},L_{\field})$ we may then glue $(\X_{\ring},\L_{\ring})$ and $(\X_{\C^*},\L_{\C^*})$ across $(X_{\field},\L_{\field})$ to obtain a family $(\X_{\C},\L_{\C}) \to \C$ satisfying $(\X_{\C},\L_{\C})|_{\ring} \cong (\X,\L)$. We may further glue $(\X,\L)$ to $(X_{\C},L_{\C})$ to obtain a family $(\X_{\pr^1},L_{\pr^1}) \to \pr^1$, with proper total space.\end{example}

This example shows that what we call models are equivalent to Donaldson's notion of an \emph{arc} for $(X,L)$ \cite{donaldson:bstab}. We reserve the terminology of \textit{arcs} for the analogous construction under a group action on projective space, as we will explain below. We use the terminology of models following the standard one in algebraic geometry, in which what we call a model for $(X,L)$ in the present article is a model for the variety $(X_{\field},L_{\field})$.

\begin{example}
\textit{Test configurations} are examples of models \cite[Definition 2.1.1]{donaldson:scalar}, \cite{tian:kems}: they are flat, $\C^*$-equivariant families $(\X,\L) \to \C$ with general fibre $(X,L)$. Much as we describe for models, it is frequently advantageous to compactify a test configuration to a family $(\X_{\pr^1},\L_{\pr^1}) \to \pr^1$ via the $\C^*$-action \cite{xiaowei-height}.
\end{example}

\begin{remark}\label{rem:filtr}Given a model $(\cX,\cL)$, the space of sections $\cV:=H^0(\cX,\cL)$ inherits the structure of a finite type $\ring$-submodule of the $\field$-vector space $V_\field:=H^0(X_\field,L_\field)$, with $\cV\otimes_{\ring}\field$ isomorphic to $V_\field$. Such a submodule is traditionally called a \textit{lattice} (see e.g.\ \cite[Section 1.3.3]{cmor:extension}, \cite[Section 1.7]{boueri}). This submodule induces a norm on $V_\field$, given by 
\begin{equation}\norm[v]_\cV=\inf\{|a|,\,a\in \field,\,v\in a\cV\}.\end{equation}
This is a vector space norm \textit{with respect to the non-Archimedean absolute value} $|\cdot|_\field$ on $\field$, and satisfying the \textit{ultrametric inequality} $$\|v+w\|_{\scV}\leq\max(\|v\|_{\scV},\|w\|_{\scV}).$$
It furthermore admits an ultrametric orthogonal basis $(v_i)_i$, in the sense that
$$\norm[\sum a_i v_i]_\cV=\max_i |a_i|_\field\norm[v_i]_\cV.$$
If $(v_i)_i$ is an $\ring$-basis of the submodule $\cV$, then it is furthermore an ortho\textit{normal} basis for $\norm_\cV$ by \cite[Lemma 1.28]{boueri}.

    This norm construction is the natural generalisation, in the setting of models, of the notion of filtration associated to a test configuration. Indeed, if $(\cX,\cL)$ is a test configuration, it classically induces a \textit{filtration} of the section ring $\oplus_{r\geq 0} H^0(X,rL)$ given by
    $$\cF^\lambda H^0(X,rL)=\{s\in H^0(X,rL),\,z^{-\lambda}\tilde s\in H^0(\cX,r\cL)\}$$
    where $\tilde s$ is the $\bbc^*$-invariant section induced by $s$ via the $\bbc^*$-action \cite{wn:tcoko}. Such filtrations yield non-Archimedean norms on each piece of the section ring $\oplus_{r\geq 0} H^0(X,rL)$, but with respect to the \textit{trivial} absolute value $|\cdot|_0$ on $\bbc$, i.e.\ the absolute value given by $|x|_0=1$ for $x\neq 0$. The associated norm in degree $r$ is given \cite[Section 1.1]{bhj:duistermaat} by
    $$\norm[s]_{\cF,r}:=e^{-\sup\{\lambda\in\bbr,\,s\in \cF^\lambda H^0(X,rL)\}}.$$

   Note that $(\bbc,|\cdot|_0)\subset (\field,|\cdot|_\field)$ is a non-Archimedean field extension, so that one can take the ground field extension of $\norm_\cF$ (see e.g.\ \cite[Definition 1.24]{boueri}) to $\field$ as a norm on $V_\field$ with respect to $|\cdot|_\field$. This norm will then coincide in each degree with the model norm induced by the test configuration $(\cX,\cL)$, seen as a model.
\end{remark}

\subsubsection{Models induced by arcs.}\label{sect:norms} We next explain how models can be constructed using the data of an embedding into projective space. Assume that $L$ is very ample, and let $\bbp^N:=\bbp (H^0(X,L))$ with $X\subset \pr^N$ embedded in projective space through global sections of $L$; the discussion is identical for models of higher exponent, where we instead embed $X$ in $\pr(H^0(X,rL))$.  Let $G:=\GL(N+1)$ be the associated general linear group, which we view as a group scheme.

\begin{definition}\label{def:arc} \cite[Section 2.1]{donaldson:bstab}
    An \textit{arc in $G$} is a $\field$-point $\rho\in G(\field)$.
\end{definition}

The following is analogous to the relationship between test configurations and one-parameter subgroups of $G$. The result is due to Donaldson \cite[Section 2.1]{donaldson:algebraic}; as it is central to our work, we report the proof. To state the result, we say that two arcs $\rho, \rho' $ are \emph{equivalent} if the element $\rho\circ\rho' \in G(\field)$ corresponds to a point of $G(\ring)$ through the inclusion $G(\ring) \subset G(\field)$ through the valuative criterion for separatedness.

\begin{proposition}\label{prop:arcsvmodels} There is a one-to-one correspondence between equivalence classes of arcs in $G$ and models of $(X,L)$ of exponent one.
\end{proposition}

\begin{proof} Suppose $\rho$ is an arc in $G$. The group scheme $G$ acts on $\pr^N$, and hence $G(\field)$ acts on $\pr^N(\field)$ through a morphism $G(\field)\times\pr^N(\field)\to\pr^N(\field)$. Thus $G$ acts on the Hilbert scheme $\Hilb_{\pr^N}$ of subschemes of $\pr^N$ with Hilbert polynomial that of $(X,L)$, and $G(\field)$ acts on $\Hilb_{\pr^N}(\field)$. From the $\field$-point $[X_{\field}]$ of $\Hilb_{\pr^N}(\field)$ induced by $[X]\in\Hilb_{\pr^N}$ we therefore obtain a point $\rho([X_{\field}]) \in\Hilb_{\pr^N}(\field)$. By the valuative criteria for properness and separatedness applied to the Hilbert scheme, we have $\Hilb_{\pr^N}(\ring) = \Hilb_{\pr^N}(\field)$, giving a point $[\X] \in \Hilb_{\pr^N}(\ring)$. More concretely, we take the flat limit of $\rho([X_{\field}])$ across zero to obtain a family $\X\to \ring$. Restricting the $\scO(1)$-line bundle from projective space to $\X$ produces a relatively very ample line bundle on $\X$. The identification of $(\X_{\field},\L_{\field})$ with $(X_{\field},L_{\field})$ is then induced by $\rho^{-1}$, producing a model.

Conversely, suppose $\pi: (\X,\L) \to \Spec\ring$ is a model of exponent $r$, so that $\pi_*\L$ is a vector bundle. We wish to produce from this input an arc in $G$, i.e.\ a morphism $\Spec \field \to G$; equivalently, we wish to produce the transition functions for a vector bundle over $\Spec \field$. Choose a trivialisation of $\pi_*\L \cong \Spec \ring \times H^0(X,L)$ of $\pi_*\L$. The isomorphism $(\X_{\field},\L_{\field})\cong (X_{\field},L_{\field})$ involved in the definition of a model induces a trivialisation of $\pi_*\L$ over $\Spec\field$, that is: an isomorphism $(\pi_*\L)_{ \field} \cong \Spec \field \times H^0(X,L)$. Thus we have two separate trivialisations of the vector bundle $\pi_*\L$ over $\Spec\field$, yielding an element of $G(\field)$, as required. This is an inverse to the previous construction, concluding the proof.\end{proof}

The result extends to models of exponent $r$ and arcs in the corresponding general linear group in the same manner. Following this result, we will pass freely between arcs and their associated models.

\subsubsection{Models induced by arcs of automorphisms.}\label{sect:normsaut} We next involve the connected component $\Aut_0(X,L)$ of the identity inside the automorphism group $\Aut(X,L)$ of $(X,L)$. This is a linear algebraic group that we may write as $$\Aut_0(X,L) = H^{\C} \ltimes U,$$ with $H^{\C}$ a maximal reductive subgroup of $\Aut(X,L)$, which is the complexification of $H$, and $U$ unipotent. Fix in addition a maximal complex torus $T^{\C}\subset H^{\C}$. 

\begin{definition}
We say a model $(\X,\L)$ is $H^{\C}$\emph{-equivariant} if there is a $H^{\C}$-action on $(\X,\L)$ making $\pi: (\X,\L) \to \Spec\field$ a $H^{\C}$-invariant morphism and such that the restriction of the action of $H^{\C}$ to $(\X_{\field},\L_{\field}) \cong (X_{\field},L_{\field})$ agrees with the action of $H^{\C}$ on $ (X_{\field},L_{\field})$. \end{definition}

Equivariant  models correspond to arcs in the centraliser $\GL(H^0(X,rL))^{H^{\C}}$, by a variant of Proposition {prop:arcsvmodels}. Equivariant models may be twisted by arcs in $H^{\C}$. For our purposes, we will only need to twist  by one-parameter subgroups $\lambda: \C^*\hookrightarrow T^{\C}$. The one-parameter subgroup $\lambda$  induces an isomorphism $X\times\C^*\to X\times \C^*$ defined by $(x,t) \to (\lambda(t).x, t)$, and hence by restriction an isomorphism which we write $$\phi_{\lambda}: X\times\Spec\field\to X\times \Spec\field.$$ Consider an $H^{\C}$-equivariant model $(\X,\L)$ of $(X,L)$, and denote by $\psi: (\X_{\field},\L_{\field}) \to (X_{\field},L_{\field})$ the associated identification.

\begin{definition}
    We define the \textit{twist} $(\cX_{\lambda},\cL_{\lambda})$ of $(\cX,\cL)$ by $\lambda$ to be $\pi: (\X,\L) \to \Spec \field$ with identification  $(\X_{\field},\L_{\field}) \to (X_{\field},L_{\field})$ given by $\phi_{\lambda}\circ\psi$.
\end{definition}

\begin{example}\label{ex:1ps}
Suppose $(\X,\L)$ is a $H^{\C}$-equivariant test configuration, with $\C^*$-action induced by $\gamma: \C^* \to \Aut(\X,\L)$. Then by hypothesis $\gamma$ and $\lambda$ commute, and the twist of $(\X,\L)$ is simply the same variety $(\X,\L)$ with $\C^*$-action $\gamma\circ\lambda = \lambda \circ \gamma$; this twisting procedure goes back to Sz\'ekelyhidi \cite{szekelyhidi-blms} and was further developed by Hisamoto \cite{hisamoto:toric} and Li \cite[Sections 3.1, 3.2]{chili:guniform}.
\end{example}

\subsection{K-stability with respect to models}\label{subsect:22} We next associate numerical invariants associated to models, \textit{via} intersection theory. We therefore first develop an intersection theory on $\X$, which requires particular care due since $\X$ is not proper. We will explain two equivalent approaches to this, beginning with an approach based on restricting to the (proper) central fibre. 

\begin{definition}\label{def:dominant}
We say that a model is \emph{dominant} if the rational map $\X \dashrightarrow X_{\ring}$ induced by the identification $ \X_{\field}\cong X_{\field}$ extends to a morphism $\X\to X_{\ring}$
\end{definition}

To any $\ring$-scheme $\X$ with an identification $\X_{\field} \cong X_{\field}$ (such as a model), we may find another $\ring$-scheme $\X'$ with suitably compatible morphisms $\sigma: \X'\to\X$ and $\X'\to X_{\ring}$, by passing to a resolution of indeterminacy of $\X \dashrightarrow X_{\ring}$. We thus assume that $(\X,\L)$ is dominant in defining various intersection numbers, with $\alpha: \X \to X_{\ring}$ the associated morphism.

Given $\Q$-line bundles $\L_0,\ldots,\L_n$ on $\X$ extending $\Q$-line bundles $L_0,\ldots,L_n$ on $X$, we wish to define the intersection number $\L_0\cdot\ldots\cdot\L_n$. For each $j=0,\ldots,n$, as the restriction $(\L_j)_{\field}$ agrees with $(\alpha^*L_j)_{\field}$ over $\X_{\field}$, we may represent their difference as $$\L_j - \alpha^*L_j = D_j$$ for a $\Q$-divisor $D_j$ supported on $\X_0$. In particular, the intersion of $n$ $\Q$-Cartier divisors with $D_j$ is well-defined, in the usual manner.

\begin{definition} We define \begin{align*}
L_0\cdot\ldots\cdot L_{n-1}\cdot \L_n := L_0|_{D_n}\cdot\ldots\cdot L_{n-1}|_{D_n},
\end{align*}
where $L_j$ denotes the pullback of $L_j$ to $\X$. By induction, we define\begin{align*}
&L_0\cdot\ldots\cdot L_{k-1}\cdot L_{k}\cdot \L_{k+1}\cdot\ldots\cdot \L_n  \\ &= L_0\cdot\ldots\cdot L_{k}\cdot L_{k+1}\cdot \L_{k+2}\cdot\ldots\cdot \L_n + L_0|_{D_{k+1}}\cdot\ldots\cdot \cdot L_{k}|_{D_{k+1}}\cdot L_{k+2}|_{D_{k+1}}\cdot\ldots\cdot \L_n|_{D_{k+1}}.
\end{align*}
\end{definition} 

This intersection number is symmetric in the $\L_j$ and independent of choice of resolution of indeterminacy, hence well-defined even when $\X$ is not dominant.

The second approach, leading to the same numerical invariant, is to canonically extend $(\X,\L_0,\ldots,\L_n)$ to $(\X_{\pr^1},\L_{0,\pr^1}, \ldots,\L_{n,\pr^1})$ following the procedure outlined in Example \ref{ex:different-models}; as $\X_{\pr}^1$ is projective we may calculate intersection numbers on $\X_{\pr^1}$ in the usual manner. The resulting values will agree with the above: $$\L_0\cdot\ldots\cdot\L_n = \L_{0,\pr^1}\cdot\ldots\cdot\L_{n,\pr^1}.$$ Thus this generalises to our model setting the usual approach to intersection numbers over test configurations  \cite{xiaowei-height, odaka2}, see also \cite[Section 6.6]{bhj:duistermaat}.

As previously mentioned, this intersection theory will allow us to define numerical invariants associated to a model $(\X,\L)$. Some of these will involve the canonical classes of $X$ and $\tilde X$; as we assume that $X$ is normal, its canonical class $K_X$ exists as a Weil divisor $K_X$. When $\X$ is further normal, its canonical class $K_{\X}$ similarly exists as a Weil divisor; if $\X$ is not normal, we pass to its normalisation  $(\tilde \X,\tilde \L)$ (with $\tilde\cL$ the pullback of $\cL$ to $\tilde X$):

\begin{lemma}If $(\cX,\cL)$ is a model of $(X,L)$, then its normalisation $(\tilde \X,\tilde \L)$ is a model for $(X,L)$.  
\end{lemma}

The proof of this result is straightforward. The Weil divisor  $K_X$ induces a Weil divison on $\X$ defined by taking the proper transform of the associated Weil divisor on $X_\ring$; this corresponds to the pullback when $\X$ is dominant and $K_X$ is $\Q$-Cartier. In addition set $\mu(X,L) = \frac{-K_X.L^{n-1}}{L^n}.$

We are now equipped to define our numerical invariants. In the following, if $\X$ is not normal, the corresponding intersection number is computed on its normalisation. We denote the canonical class of $\X$ by $K_{\X}$, the log canonical class of $\X$ by $$K_{\X}^{\log} = K_{\X} + \X_{0,\red} - \X_0,$$ with $\X_{0,\red}$ the reduction of the central fibre $\X_0$, and $$K_{\X/\pr^1} = K_{\X} - \pi^*K_{\pr^1}, \qquad K_{\X/\pr^1}^{\log} = K_{\X}^{\log} - \pi^*K_{\pr^1}$$ their relative versions.

\begin{definition} We define the:
\begin{enumerate}[(i)]
\item \emph{Donaldson--Futaki invariant} of $(\X,\L)$ to be $$\DF(\X,\L) = \frac{1}{L^n}\left(\frac{n}{n+1}\mu(X,L)\L^{n+1} + \L^n.K_{\X/\pr^1}\right);$$
\item  \emph{non-Archimedean Mabuchi functional} of $(\X,\L)$ to be \begin{align*}
\mM\NA(\X,\L)&=\frac{1}{L^n}\left(\frac{n}{n+1}\mu(X,L)\L^{n+1} + \L^n.K^{\log}_{\X/\pr^1}\right)\\
&=\mH\NA(\X,\L) + \mR\NA(\X,\L)+n(-K_X.L^{n-1})\mE\NA(\X,\L),
\end{align*}
where
\begin{enumerate}[(a)]
	\item $\mE\NA(\X,\L) = \frac{(\L^{n+1})}{(n+1)L^n};$
 	\item $\mR\NA(\X,\L) = \frac{1}{L^n}(\L^n.K_X);$
 \item $\mH\NA(\X,\L) =\frac{1}{L^n}(\L^n.(K_{\X/X\times \pr^1}^{\log}))$.
\end{enumerate}
\end{enumerate}
Note that $\mM\NA(\X,\L) \leq \DF(\X,\L)$, with equality if and only if $\X_0$ is reduced.
    
\end{definition}

Both the Donaldson--Futaki invariant and the non-Archimedean Mabuchi functional will play prominent roles in our work. The Donaldson--Futaki invariant is the generalisation of Donaldson and Wang \cite{donaldson:bstab, xiaowei-height} to arcs of the usual Donaldson--Futaki invariant \cite{donaldson:scalar,tian:kems}, with the non-Archimedean Mabuchi functional generalising to arcs the invariant introduced by Boucksom--Hisamoto--Jonsson \cite{bhj:duistermaat}.

\begin{remark}Each of these terms may be defined explicitly by intersection numbers supported on the central fibre. For example, writing $\L - L = D$ for $D$ supported on $\X_0$ as in our definition of intersection numbers, a simple calculation shows that $$\mE\NA(\X,\L) = \frac{(\L^{n+1})}{(n+1)L^n} =  \frac{1}{(n+1)L^n}\sum_{j=0}^n (\L^j.L^{n-j})|_D.$$
\end{remark}

We also require versions of ``norms'' on the space of models. Throughout, we denote by $N_{\Z}$ the cocharcater lattice of $T^{\C}$, so the lattice of one-parameter subgroups. 

\begin{definition}
We define the  \emph{norm} of $(\X,\L)$ to be 
$$\|(\X,\L)\| = \frac{\L.L^n}{L^n} -\mE\NA(\X,\L) .$$ When $(\X,\L)$ is $H^{\C}$-equivariant, we define the \emph{reduced norm} of $(\X,\L)$ to be 
 $$\|(\X,\L)\|_{T^{\C}} = \inf_{\lambda \in N_{\Z} }\|(\X_{\lambda},\L_{\lambda})\|.$$
\end{definition}

The norm extends the minimum norm of \cite{uniform-twisted} and the non-Archimedean $J$-functional of \cite{bhj:duistermaat} in the case of test configurations, with the reduced norm originating in work of Hisamoto \cite{hisamoto:toric}. These definitions allow us to define the  stability conditions relevant to our work.

\begin{definition} We say that $(X,L)$ is 
\begin{enumerate}[(i)]
    \item \emph{uniformly K-stable with respect to models} if there is an $\epsilon>0$ such that for all models $(\X,L)$ $$\DF(\X,\L) \geq \epsilon \|(\X,\L)\|;$$ 
    \item \emph{uniformly K-polystable with respect to models} if there is an $\epsilon>0$ such that for all $H^{\C}$-equivariant models $(\X,L)$ $$\DF \geq \epsilon \|(\X,\L)\|_{T^{\C}}.$$
\end{enumerate}
\end{definition}

K-semistability is similarly defined by requiring nonnegativity of the Donaldson--Futaki invariant.

\begin{remark}\label{DFMNA}These conditions imply uniform K-stability and uniform K-stability with respect to test configurations, respectively. A simple base change argument implies that the Donaldson--Futaki invariant may be replaced by the non-Archimedean Mabuchi functional in these definitions, much as for test configurations \cite[Proposition 8.2]{bhj:duistermaat}. \end{remark}

\section{A numerical criterion for stability of pairs}\label{sec:pairs}

We change focus and consider Paul's theory of stability of pairs. We begin by discussing the role of arcs in the setting of a pair of group representations, before producing a numerical criterion for Paul's notion of a semistable or stable pair. We then introduce a notion of a polystable pair, and give numerical and analytic characterisations of polystability. In Section \ref{sec:applications} we will then apply the results proven here to K-stability.

\subsection{Arcs and weights}

We consider an algebraic group $G$ acting linearly on a pair of vector spaces $V,W$, and fix a point $(v,w) \in V\oplus W$ with neither $v$ nor $w$ equal to zero. We write $[v,w]$ for the corresponding point in $\bbp(V\oplus W)$. Our motivation is to understand the following condition, whose definition is due to Paul.

\begin{definition}\cite{paul:hyperdiscriminants} We say that the pair $(v,w)$ is \emph{semistable} if $$\overline{G.[v,w]} \cap \pr(0\oplus W) = \varnothing.$$
\end{definition}

\begin{example}\label{ex:git-through-pairs} Suppose $V=\C$ endowed with the trivial $G$-action, and let $v=1$. Then the pair $[v,w]$ is semistable if and only if $0\notin \overline{G.w}$, which is equivalent to $[w]\in\pr(W)$ being semistable in the sense of geometric invariant theory (GIT) \cite{book:git}. Thus semistability of pairs is a generalisation of GIT semistability that involves two $G$-representations rather than one. 
\end{example}

In the subsequent section, we will give a numerical criterion for semistability through arcs. Before doing so, we establish some basic results involving numerical invariants associated to arcs.  Recall from Definition \ref{def:arc} that an arc $\rho$ in $G$ is simply a $\field$-point of $G$, or equivalently a morphism $\Spec\field \to G$. From the action of the group scheme $G$ on $ \pr(V\oplus W)$, we obtain an action of $G(\field)$ on  $(V\oplus W)(\field)$ and $ \pr(V\oplus W)(\field)$ and hence we may consider the $\field$-points $\rho.(v,w)_\field\in (V\oplus W)_\field$ and $\rho.[v,w]_{\field} \in \pr(V\oplus W)(\field)$, with  (for example) $[v,w]_\field$ being the constant morphism $\Spec\field\to \bbp(V\oplus W)$ associated to $[v,w]$, as used in Proposition \ref{prop:arcsvmodels}.

Concretely, we may choose  coordinates $[x_0,\ldots,x_p,y_0,\ldots,y_q]$ for $V\oplus W$ where $\dim V = p+1$ and $\dim W=q+1$. Then, as the field  $\field=\bbc\llparenthesis z\rrparenthesis$ consists of Laurent series, we may write
\begin{equation}\label{eq:weight}
\rho.(v,w)_\field= \sum_{i=0}^{p} a_i\cdot  (x_i)_\field + \sum_{i=0}^{q} b_i\cdot (y_i)_\field
\end{equation}
with $$a_i=a_{0,i} z^{\ord_0(a_i)} + O\left(z^{\ord_0(a_i)}+1\right), \qquad b_i=b_{0,i} z^{\ord_0(b_i)} + O\left(z^{\ord_0(b_i)}+1\right).$$

\begin{definition}\label{def:weight}

We define the \emph{weight} $\nu(\rho,[v,w]) $ of the arc $\rho$ to be $$\nu(\rho,[v,w]) = \min_{i\in \{0,\ldots q\},\,b_i\neq  0}\ord_0(b_i)  - \min_{i\in \{0,\ldots, p\},\,a_i\neq 0}\ord_0(a_i) \in \Z.$$

\end{definition}

In the case of a single vector, this is precisely the invariant introduced by Donaldson \cite[Section 2.1]{donaldson:bstab}. Justifying the notation, note that the weight is independent of the choice of representative of $[v,w]_\field\in \bbp(V\oplus W)_\field$ in $(V\oplus W)_\field$, as scaling $(v,w)_\field\in (V\oplus W)_\field$ by a Laurent series $c_0 z^c + O(z^{c+1})$ will change the weights for $v$ and $w$ both by $c$, leaving their difference unchanged. One further checks this definition to be independent of choices of coordinates for $V$ and $W$. Lastly, we remark that our definition makes sense for an arbitrary element $(v,w)_\field\in V_\field\oplus W_\field$, rather than just the field extension of a pair of vectors in $V\oplus W$, although this will not be needed in our applications.

\begin{example}\label{weight-in-GIT}
For an arc induced by a one-parameter subgroup $\lambda \hookrightarrow G$, and a single vector space $V$ (so taking $W=\C$ with the trivial action as in Example \ref{ex:git-through-pairs}), the weight is simply the usual weight attached to a one-parameter subgroup in GIT \cite[Section 2.1]{book:git}. 
\end{example}

\begin{remark}Consider the \textit{trivial norms} on $V$ and $W$, which are equal to $1$ on $V^\times$ and $W^\times$. These are non-Archimedean norms with respect to the trivial absolute value $|\cdot|_0$ on $\bbc$ (recall Remark \ref{rem:filtr}), and one can take their ground field extension $\norm_{\triv}$ to $V_\field$ and $W_\field$. This extension is characterised as follows: for all bases $(v_i)_i$ of $V$,
$$\norm[\sum a_i (v_i)_\field]_\triv=\max_{i,\,a_i\neq 0} |a_i|_\field;$$
in other words, it admits all (ground field extensions of) bases of $V$ as orthonormal bases; likewise for $\norm_\triv$ on $W$. As in Equation \eqref{eq:weight} we may consider $\rho.(v)_\field\in V_\field$ and $\rho.(w)_\field\in W_\field$; from the fact that $|\cdot|_\field=e^{-\ord_0(\cdot)}$, we then see that
$$\nu(\rho,[v,w])=\log\frac{\norm[\rho.(w)_\field]_\triv}{\norm[\rho. (v)_\field]_\triv}.$$
\end{remark}

The following lemma, which follows from simple calculations, explains the basic properties of the weight and will be useful later:

\begin{lemma}\label{lem:rules}Let $V,V'$ and $W,W'$ be representations of $G$, and let $v,v',w,w'$ belong to each respective vector space. Then,
$$\nu(\rho,[v\otimes v',w\otimes w'])=\nu(\rho,[v,w])+\nu(\rho,[v',w']).$$
	In particular,
	\begin{enumerate}[(i)]
	\item $\nu(\rho,[v,v])=0$;
	\item $\nu(\rho,[v\otimes v',w\otimes v'])=\nu(\rho,[v,w])$;
	\item $\nu(\rho,[v,w])=\nu(\rho,[v,v'])+\nu(\rho,[v',w])$; 
	\item $\nu(\rho,[v,w])=-\nu(\rho,[w,v])$;
	\item $\nu(\rho,[v^{\otimes k},w^{\otimes k}])=k\nu(\rho,[v,w])$.
	\end{enumerate}
\end{lemma}

We may also interpret the weight of an arc as in ``intersection number'', as follows. Since $G$ acts on $V$ and $W$, we obtain from an arc $\rho$ a pair of $\field$-points $\rho.[v]_{\field}\in\pr(V)_\field$ and $\rho.[w]_\field\in\pr(W)(\field)$. Since projective space satisfies $\pr(V)({\field}) = \pr(W)({\ring})$ by separatedness and properness (and similarly for $\pr(W)$),   they extend to points $\mathfrak{v}\in\bbp(V)(\ring)$ and $\mathfrak{w}\in\bbp(W)(\ring)$, yielding by restriction of $\cO(1)$ a pair of polarised $\ring$-points $(\mathfrak{v},\scO_{\pr(V)}(1))$ and $(\mathfrak{w},\scO_{\pr(W)}(1))$. By construction there is an identification with restriction to $\Spec \field$ $$(\mathfrak{v},\scO_{\pr(V)_\ring}(1)) \cong (\rho.[v]_{\field},\scO_{\rho.[v]_{\field}})$$ with $\scO_{\rho.[v]_{\field}}$ the structure sheaf (namely the trivial line bundle); and similarly for $\mathfrak{w}$. 

Focusing on $(\mathfrak{v},\scO_{\pr(V)_\ring}(1))$ for the moment, we may assume that we have a morphism $\pi: \mathfrak{v} \to \rho.[v]_{\field}$ extending the identification above. Thus we may write  $$\scO_{\pr(W)_\ring}(1) - \pi^*(\scO_{\rho.[v]_{\field}}) = D_\mathfrak{v}$$ for $D_\mathfrak{v}$ a divisor supported in $\mathfrak{v}_0$; we will be interested in the degree of this divisor (as a zero-dimensional version of an intersection number). A similar process produces a divisor $D_\mathfrak{w}$ supported in $\mathfrak{w}_0$. 

\begin{lemma}\label{weight-as-intersection}
The weight of the arc $\rho$ satisfies $$\nu(\rho,[v,w]) = \deg D_\mathfrak{v} - \deg D_\mathfrak{w}.$$ 
\end{lemma}

\begin{proof}

We construct the divisor $D_\mathfrak{v}$ in more detail. Take a trivialising section $s$ of $\scO_{\rho.[v]_{\field}}$ and consider its pullback $\pi^*s$ to a section of $\pi^*(\scO_{\rho.[v]_{\field}})$ over $\mathfrak{v}$. Choosing a coordinate $z$ on $\mathfrak{v}$, there is an integer $m_\mathfrak{v}$ such that $z^{m_\mathfrak{v}}\pi^*s$ is a trivialising section of $\scO_{\pr(V)_\ring}(1)$ over $\mathfrak{v}$, and the divisor $D_\mathfrak{v}$ is simply the divisor associated to $z^{m_\mathfrak{v}}$, namely $m_\mathfrak{v}[0]$. It is the smallest integer such that  $z^{m_\mathfrak{v}}\pi^*s$ is regular. It follows then from the construction that in this notation, $ \deg D_\mathfrak{v}$ equals $m_\mathfrak{v}$.

 The arc $\rho$ induces a $\field$-point of $\pr(V)$ which is explicitly $\rho.[v]_\field=[a_0:\dots:a_k]$ in the notation of Definition \ref{def:weight} (the $a_i$ being elements of $\field$, so Laurent series), so that this smallest integer $m_\mathfrak{v}$ is $-\min_i (\ord_0(a_i))$. We likewise obtain from $w$ an integer $m_\mathfrak{w}=-\min_i(\ord_0(b_i))$. Using the definition of the weight and the fact that $\deg D_\mathfrak{v}=m_\mathfrak{v}$ concludes the proof. \end{proof}

\begin{remark} This result implies that the weight of an arc generalises also a numerical invariant introduced by Wang \cite[Definition 3]{xiaowei-height}, which he calls the height. Wang considers a single point, rather than a pair, but considers his height as attached to sections of certain fibre bundles, as a version of GIT in families. If one considers an arc as inducing a section of the trivial family $\pr(V) \times\Spec \ring \to \Spec \ring$, then the resulting numerical invariant agrees with that defined by Wang.\end{remark}

\subsection{Semistability of pairs} We next relate semistability to the following numerical condition.

\begin{definition}

We say that $[v,w]$ is \emph{numerically semistable} if for all arcs $\rho$ in $G$ we have $\nu(\rho,[v,w])\geq 0$.

\end{definition}

This is all we need to prove our first main result.

\begin{theorem}\label{thm:HMcriterion} The pair $[v,w]$ is semistable if and only if it is numerically semistable.
\end{theorem}

\begin{proof}

Suppose first that $[v,w]$ is semistable and let $\rho\in G(\field)$ be an arc. The identification $\pr(V\oplus W)(\field) = \pr(V\oplus W)(\ring)$ means we may take the specialisation $[\tilde v,\tilde w]$ of $[v,w]$ under $\rho$ (namely the image of $0 \in \Spec \ring$ of the associated $\ring$-point). By construction, $[\tilde v,\tilde w] \in  \overline{G.[v,w]}$ and in particular $[\tilde v,\tilde w] \notin \pr(0\oplus W)$ by semistability. 

We claim that the condition $[\tilde v,\tilde w] \notin \pr(0\oplus W)$ is equivalent to the weight being nonnegative. Let $m_\mathfrak{v}$ and $m_\mathfrak{w}$ be as in the proof Lemma \ref{weight-as-intersection}, so that $\nu(\rho,[v,w]) = m_\mathfrak{v}-m_\mathfrak{w}$. We may multiply the homogeneous coordinates of $\rho.[v,w]_\field$ by $z^{m_\mathfrak{w}}$ (which leaves the point in $\pr(V\oplus W)$ unchanged), yielding entries $$[z^{m_\mathfrak{w}}a_0:\ldots:z^{m_\mathfrak{w}}a_k:z^{m_\mathfrak{w}}b_0:\ldots:z^{m_\mathfrak{w}}b_l].$$
(Here, again, the $a_i$ and $b_i$ are Laurent series in $\field$ and not \textit{a priori} constant coefficients.) From this one can see that, for the specialisation of $\rho.[v,w]$ to lie in $\bbp(0\oplus W)$, the orders of vanishing of the $V$-coordinates must be positive: $\ord(z^{m_\mathfrak{w}}a_i)>0$ for all $i=0,\ldots,k$, or in other words
$$\min_{i\in\{0,\dots,k\},\, a_i\neq 0}\ord_0(z^{m_\mathfrak{w}}a_i)=m_{\mathfrak{w}}+\min_{i\in\{0,\dots,k\},\, a_i\neq 0}\ord_0(z^{m_\mathfrak{w}}a_i)=m_\mathfrak{w}-m_\mathfrak{v}>0.$$
Thus the specialisation of $\rho.[v,w]$ lies in $\pr(0\oplus W)$ if and only if $\nu(\rho,[v,w])  = m_\mathfrak{v}-m_\mathfrak{w} < 0$, proving the claim and hence one direction of the result.

In the reverse direction, suppose $[v,w]$ is not semistable,  meaning by definition that $$\overline{G.[v,w]} \cap  \pr(0\oplus W) \neq \varnothing,$$ and fix $[\tilde v,\tilde w]\in  \overline{G.[v,w]} \cap  \pr(0\oplus W).$ By \cite[Section 4.2]{hoskins} (which we describe the history of further in Remark \ref{rem:hoskins} below), we may find an arc $\rho \in G$ such that $[v,w]$ specialises to $[\tilde v,\tilde w]$. But by the calculation of the previous paragraph, the condition  $[\tilde v,\tilde w] \in \pr(0\oplus W)$ forces $\nu(\rho,[v,w])< 0$, proving the result.  \end{proof}

\begin{example}\label{PSS-example}  By Example \ref{weight-in-GIT}, the weight of an arc induced by a one-parameter subgroup is simply the usual one involved in GIT. For a single representation by a reductive group, the Hilbert--Mumford criterion proves that GIT semistability of $[w]$ is equivalent to the numerical condition that $\nu(\lambda, [w])\geq 0$ for all one-parameter subgroups $\lambda$ of $G$. So Theorem \ref{thm:HMcriterion} can be thought of as a numerical criterion that involves arcs alongside one-parameter subgroups, but which applies more generally to $G$-actions on pairs of vector spaces, hence to differences of ample line bundles, and which further does not require reductivity. 

By work of Paul--Sun--Zhang \cite{paulsunzhang}, this result is optimal: semistability of pairs (for reductive groups) cannot be characterised by one-parameter subgroups alone. Paul--Sun--Zhang give several counterexamples, with the simplest involving an $\SL(2,\C)$-action on $(\Sym^n\C^2) \oplus (\Sym^n \C^2\oplus \Sym^{n-2}\C^2)$.
\end{example}

\begin{remark}\label{rem:hoskins} The existence of a ``destabilising'' arc in the proof of Theorem \ref{thm:HMcriterion} is central to our work. The existence of such an arc is a key ingredient in Mumford's proof of the Hilbert--Mumford criterion \cite[Section 2.1]{book:git}, and is explained in detail by Hoskins  \cite[Section 6.4]{hoskins}. It is also involved in one proof of the valuative criterion for properness \cite[Section 7.2]{kempf}, as mentioned by Boucksom--Hisamoto--Jonsson \cite[Section 5.1]{bhj:asymptotics}. The proof is essentially by a dimension-counting, genericity argument. 

This construction of an arc is the first step in Mumford's proof of the Hilbert--Mumford criterion. The remaining steps produce a one-parameter subgroup $\lambda$ of $G$ associated to the arc $\rho$ using the Cartan--Iwahori decomposition, and the final step of the Hilbert--Mumford criterion then shows that the specialisation of $[v,w]$ under $\lambda$ is $[\tilde v,\tilde w]$, which is why $\lambda$ destabilises. It is this last step that fails in the examples of Paul--Sun--Zhang described in Example \ref{PSS-example}. \end{remark}

We next turn to a more analytic aspect of the theory of stability of pairs, due to Paul \cite{paul:hyperdiscriminants}. We will later wish to extend these results to more general settings, and so recall them here. Endow $V$ and $W$ with Hermitian inner products, so that for each $g\in G$ we obtain values $|g.v|$ and $|g.w|$ associated to $(v,w)\in V\oplus W$. We consider the \emph{log norm functional} $f: G\to \R$ defined by $$f(g) = \log|g.v| - \log|g.w|.$$

\begin{definition} 
We say that $[v,w]$ is \emph{analytically semistable} if there is a constant $\delta>0$ such that for all $g\in G$ we have $f(g)\geq -\delta$.
\end{definition}

The following is the key analytic result surrounding semistability of pairs.

\begin{theorem}[{\cite[Proposition 2.6]{paul:hyperdiscriminants}}]\label{thm:paullognorm}
A pair $[v,w]$ is semistable if and only if there is a constant $\delta>0$ such that for all $g\in G$ we have $f(g)\geq -\delta$.
\end{theorem}

\begin{remark}Paul proves this by establishing  the formula $$\log \tan^2 \dist_{\FS}(G.[v,w],G.[v,0]) = \inf_{g\in G}\{\log|g.v| - \log|g.w|\},$$ where the distance $\dist_{\FS}(G.[v,w],G.[v,0])$ is measured through the Fubini--Study metric on $\pr(V\oplus W)$ induced by the inner product on $V\oplus W$ and is calculated as the infimum of the distance of any two points in $G.[v,w]$ and $G.[v,0]$. Boucksom--Hisamoto--Jonsson \cite[Lemma 5.5]{bhj:asymptotics} have given another proof.\end{remark}

We next consider the behaviour of the log norm functional $f$ along an arc. By definition, $\rho:\Spec \field \to G$ induces a morphism $\Spec \field \to V$ through the action on $v$, and hence a morphism $\Spec\ring\to V$; likewise for $W$.  As we will be interested in asymptotics in the analytic topology, it will be more transparent to canonically extend this morphism to a morphism $\Spec\C[z] \to V$, following the process of Example \ref{ex:different-models}, as we may then endow the affine line $\C$ with the analytic topology, with coordinate $z$. 

\begin{lemma}\label{lem:slope} We have 

$$f(\rho(z))=\nu(\rho,[v,w])\log|z|\mi + O(1).$$

\end{lemma}

\begin{proof}

As before we write explicitly $\rho.[v,w]_\field=[a_0:\dots:a_p:b_0:\dots:b_q]$ with the $a_i,b_i\in \field$ in chosen coordinates as before.  Considering the $V$-coordinates to begin with, we see 
\begin{align*}
2\log|\rho(z). v| &= \log \left(\sum_{i=0}^p\left|  a_{0,i}z^{\ord_0(a_i)}+O(z^{\ord_0(a_i)+1})\right|^2\right)\\
&= \max_{i,a_i\neq 0}(\log |z|^{2\ord_0(a_i)})+O(1)\\
&=\max_{i,a_i\neq 0}(-2\ord_0(a_i))\log|z|\mi+O(1)\\
&=-\min_{i,a_i\neq 0}(2\ord_0(a_i))\log|z|\mi+O(1) \\
&=2m_\mathfrak{v}\log|z|\mi+O(1).
\end{align*}
Likewise we obtain
$$2\log|\rho(z). w|=2m_\mathfrak{w}\log|z|\mi+O(1)$$
and the result follows as $\nu(\rho,[v,w])=m_\mathfrak{v}-m_\mathfrak{w}$.        \end{proof}
         
         We view $\nu(\rho,[v,w])$ as the slope of the log norm functional $f$ along the arc $\rho$.

\begin{example} In the special case that $\rho$ is induced by a one-parameter subgroup, our result recovers Boucksom--Hisamoto--Jonsson \cite[Theorem 5.4(i)]{bhj:asymptotics}.\end{example}

The following summarises the results of this section.

\begin{corollary}\label{cor:body-HNKM} For a pair $[v,w]$, the following are equivalent:
\begin{enumerate}[(i)]
\item $[v,w]$ is semistable;
\item $[v,w]$ is numerically semistable;
\item $[v,w]$ is analytically semistable.
\end{enumerate}
\end{corollary}

By our slope formula, these are further equivalent to the log norm functional $f$ having nonnegative slope at infinity along all arcs. 

\subsection{Stability of pairs} The results we have proven for semistability directly apply further to Paul's notion of a \emph{stable pair}, for which we need another $G$-representation. As the group $G$ is a linear algebraic group,  we may embed it into $\GL(m)$ for some $m$, giving a natural $G$-action on $\C^{m\times m}$. Denote by $e\in\GL(m) \subset \C^{m\times m}$ the identity matrix, and let $\deg V$ denote the degree of the $G$-representation $V$ in the sense of \cite[(2.6)]{paul:hyperdiscriminants}, that is: the smallest positive integer $d$ such that the weight polytope of all nonzero $v\in V$ lie in the rescaled simplex $d\Delta_{\dim V}$.

\begin{definition}
We say that $[v,w]$ is \emph{stable} if there exists an integer $k>0$ such that the pair $$[e^{\otimes \deg V}\otimes v^{\otimes k}, w^{\otimes k+1}] \in \pr(((\C^{m\times m})^{\otimes \deg V} \otimes V^{\otimes k}) \oplus W^{\otimes k+1})$$ is semistable.
\end{definition}

The definition is motivated by connections to K\"ahler geometry, as we shall see in Section \ref{sec:applications}. 

We next explain the  analogue of the numerical criterion.  Defining the \emph{norm} of $\rho$ to be
\begin{equation}
\|(\rho,[v])\|:=\nu(\rho,[v,e^{\otimes \deg V}]),
\end{equation}
a simple computation using Lemma \ref{lem:rules}  shows the weight of the point $[e^{\otimes \deg V}\otimes v^{\otimes k}, w^{\otimes k+1}] $ under $\rho$ equals $(k+1)(\nu(\rho,[v,w])) - \|(\rho,[v,w])\|.$ By Tian \cite[Lemma 5.1]{tian:survey} (who uses the definition of $\deg V$), it follows that the norm of an arbitrary arc is nonnegative. 

\begin{definition} We say that $[v,w]$ is \emph{numerically stable} if there exists an integer $k>0$ such that for all arcs $\rho$ we have $\nu(\rho,[v,w]) \geq (k+1)^{-1} \|(\rho,[v,w])\|.$
\end{definition}

It then automatically follows from Theorem \ref{thm:HMcriterion} that numerical stability is equivalent to stability. It will also be useful in our subsequent work on polystability to point out that this is further equivalent to  stability of the pair $[e^{\otimes \deg V}\otimes v^{\otimes k+1},w^{\otimes k+1}\otimes v]$ (as in the calculation above, using Lemma \ref{lem:rules}(2), simply tensoring by $v$).

We similarly obtain a log norm interpretation from Theorem \ref{thm:paullognorm}, due to Paul. 

\begin{definition}
We say that $[v,w]$ is \emph{analytically stable} if there exists an integer $k>0$ and $\delta>0$ such that  for all $g\in G$ we have $$ \log|g.v| - \log|g.w| \geq (k+1)^{-1} ((\deg V)\log|g| - \log|g.v|) -\delta.$$
\end{definition}

 Fixing Hermitian inner products on $V, W$ and $\C^{m\times m}$ we obtain an induced inner product on the vector space $(\C^{m\times m})^{\otimes \deg V} \otimes V^{\otimes k} \oplus W^{\otimes k+1}$, where we recall that the norm of a tensor product is the product of the norms. Summarising, we obtain:

\begin{corollary}\label{coro:stability} The following are equivalent:
\begin{enumerate}[(i)]
\item the pair $[v,w]$ is stable;
\item the pair $[v,w]$ is numerically stable;
\item the pair $[v,w]$ is analytically stable.\end{enumerate}
\end{corollary}

\subsection{Polystability of pairs} We incorporate automorphisms of $[v,w]$ and introduce a notion of a polystable pair, a numerically polystable pair and an analytically polystable pair, with the main result being that all three conditions are equivalent.

\begin{definition}
The \emph{automorphism group} of $[v,w]$ is the intersection of stabilisers $$\Aut([v,w]) = G_{[v]}\cap G_{[w]}.$$
\end{definition}

Write $$G_{[v]}\cap G_{[w]} = H^{\C} \ltimes U,$$ with $H^{\C}$ a maximal connected reductive subgroup and $U$ unipotent, and fix a maximal complex torus $T^{\C}$ lying inside the identity component of $H^{\C}$. The group $H^{\C}$ acts on $V\oplus W$, and we consider the fixed locus $(V\oplus W)^{H^{\C}} = V^{H^{\C}} \oplus W^{H^{\C}}$, which is a linear subspace as $\Aut([v,w])$ acts linearly. Note that $(v,w) \in (V\oplus W)^{H}$. We fix the groups $H, T^{\C}$ throughout the current discussion. Denote further $G^H$ the centraliser of $H^{\C}$ in $G$, so that there is an action of $G^H$ on $(V\oplus W)^{H^{\C}}$, and so that $T^{\C} \subset G^H$.

We begin by defining polystability of a pair. Similarly to the stable case, for $k>0$ we will be interested in the geometry of the projective space $$\pr^N = \pr\left(((\C^{m\times m})^{\otimes \deg V} \otimes V^{\otimes k+1})^{H^{\C}} \oplus (W^{\otimes k+1} \otimes V)^{H^{\C}}\right)$$ and the point $[e^{\otimes \deg V}\otimes v^{\otimes k+1}, w^{\otimes k+1}\otimes v]$ lying in it, associated to a pair $[v,w]$. The group $G^H$ acts in the natural way on this space, and we further let $T^{\C}$ act nontrivially only on the $(\C^{m\times m})^{\otimes \deg V}$ component and the $V$-component of the second factor in the direct sum---namely $W^{\otimes k+1} \otimes V$---giving an action of $G^H\oplus T^{\C}$ on $\pr^N$. 

For $u\in \pr^N$, let $C_u = \overline{T^{\C}.u},$ which is a subvariety of $\pr^N$. Then for $g\in G^H$, since $g$ commutes with $T^{\C}$ we have $$g\left(T^{\C}.u\right) = T^{\C}.g(u)$$ and hence the same is true of their closures: $C_{g(u)} =g(C_u).$ The $G^H$-action induces an action on the Hilbert scheme $\Hilb$ parameterising subschemes of $\pr^N$ with the same Hilbert polynomial of $C_p$; these polynomials are all equal since $C_{g(u)} =g(C_u)$.

Let $\Hilb_{\des}\subset \Hilb$ denote the subscheme of $\Hilb$ corresponding to subschemes entirely contained in $\pr(0\oplus (W^{\otimes k+1}\otimes V)^{H^{\C}})$; this is the ``destabilising'' part of the Hilbert scheme.

\begin{definition}\label{def:polystable} We say that $[v,w]$ is \emph{polystable} if there exists a $k>0$ such that, setting $u = [e^{\otimes \deg V}\otimes v^{\otimes k+1}, w^{\otimes k+1}\otimes v]$,  we have 

 $$\overline{G^H.[C_u]}\cap \Hilb_{\des} = \varnothing.$$
\end{definition}

\begin{remark}
If a point $u$ specialises to a point $u'$ through an arc  $\rho$ in $G^H$, then the dimension of $C_u$ may be strictly smaller than the dimension of $C_{u'}$, as the stabiliser may increase in the limit. So one should not expect $C_{u'}$ to be the limit of the $C_u$ in $\Hilb$ in general, as $[C_{u'}]$ may even lie in a different Hilbert scheme. This phenomenon is related to what Donaldson calls ``splitting of orbits'' \cite[p.\ 115]{donaldson:algebraic}.
\end{remark}

We wish to compare this to a condition involving log norm functionals, and a condition involving arcs. We begin by giving the various definitions. Numerical polystability involves twisting an arc in $G^H$ by a one-parameter subgroup of $T^{\C}$, which is defined in exactly the same way  as was done for models in Section \ref{sect:normsaut} for models.

\begin{definition}We define the (reduced) \emph{norm} of an arc $\rho\in G^H(\field)$ to be 
$$\|(\rho,[v])\|_{T^{\C}} = \inf_{\lambda \in N_{\Z}} \|(\rho_{\lambda},[v])\|,$$ where $\rho_{\lambda}$ is the twist of $\rho$ by an arbitrary one-parameter subgroup $\lambda$ of $T^{\C}$, and say that $[v,w]$ is \emph{numerically polystable} if there exists a $k>0$ such that $$\nu(\rho, [v,w]) \geq (k+1)\mi \|(\rho,[v])\|_{T^{\C}}.$$\end{definition}

Turning to the analytic counterpart, as all results in the current section rely on the following property, we state it explicitly:

\begin{definition}\label{def:proper} We say that the \textit{norm functional} is \emph{proper} if for all $g\in G^H$ the function $T^{\C}\to \R$ defined by $$t\mapsto (\deg V)\log|t.g.e| - \log|t.g.w|$$ is proper. \end{definition}

In applications, we will show that this property holds. The analytic version of polystability is then as follows. 

\begin{definition}\label{def:anal-poly}
We say that a pair $[v,w]$ is \emph{analytically polystable} if there exist $\epsilon, \delta>0$ such that for all $g\in G^H$ $$\log|g.v| - \log|g.w| \geq \epsilon \inf_{t\in T^{\C}}((\deg V)\log|t.g.e| - \log|t.g.v|) - \delta.$$ 
\end{definition}

We next allow twists of arcs by elements of $N_{\R} = N_{\Z} \otimes \R$, in essence allowing twists by irrational elements of $\Lie T^{\C}$. To make sense of this, first note that given $\xi \in N_{\R}$ we may consider for $s\in \C$ the exponential $\exp(s\xi)\in T^{\C}$, and so we may consider $\rho(\exp(s))\circ \exp(s\xi)\in G^H$. In particular, just as for arcs we may define $\nu(\rho_{\xi},[v,w])$ by considering the behaviour of the order of pole of $\rho_{\xi}.(v,w)_\field \in (V\oplus W)(\field)$.

\begin{lemma}\label{lem:continuity} Fix an arc $\rho \in G^H(\field)$, and suppose the norm functional is proper. Then:
\begin{enumerate}[(i)] 
\item the function $N_{\R} \to \R$ given by $\xi \to \nu(\rho_{\xi},[v,w])$ is continuous;
\item the expansion $$f(\rho(\exp(-s))\circ \exp(s\xi))=|s|\nu(\rho_{\xi},[v,w]) + O(1)$$ holds as $|s|\to \infty$.
\end{enumerate}
\end{lemma}

The proofs are direct calculations, much as Lemma \ref{lem:slope}. We next prove the following, which is in essence an extension of work of Li from test configurations to arcs \cite[Theorem 3.14]{chili:guniform}. The statement applies to arbitrary pairs; since we are interested in applying it to the norm, we will denote the pair by $[e,v] \in \pr(E\oplus V)$ for vector spaces $E,V$,  since this is essentially what appears in applications (when the weight becomes, by definition, the norm).

\begin{proposition}\label{prop:slopeequivfindim}Assume that the norm functional is proper. Then a $\xi\in N_{\R}$ infimising $\nu_{T^\bbc}(\rho,[e,v])$ exists and we have
$$\inf_{t\in T^{\C}}(\log|t.\rho(z).e| - \log|t.\rho(z).v|)=\nu_{T^\bbc}(\rho,[e,v])\log|z|\mi + O(1).$$
\end{proposition}

\begin{proof}
We begin by showing that there is an element of $N_{\R}$ that computes the infimum asymptotically. By properness of the norm, for each $z$  there exists a $t_z\in T^\bbc$ such that
$$f(t_z)=\inf_{t\in T^{\C}}(\log|t.\rho(z).e| - \log|t.\rho(z).v|)=:f_{T^\bbc}(z).$$
Thus there exists for each $z$ a $\xi_z\in N_{\Z}$ with $t_z=\xi_z(z)$ (namely, we view $\xi$ as giving a map $\C^*\to G$ and we evaluate this map at $z\in \C^*$). Throughout we will use the notation $\xi(z)$, for $z\in \C^*$, to mean $\xi(e^{-\log|z|\mi})$ in the usual sense. 

We claim that the $\xi_z$ are uniformly bounded in $N_{\R}$. Indeed, along the arc $\rho$ there exists $C>0$ such that
$$\log|\rho(z).e| - \log|\rho(z).v|\leq C\log|z|\mi$$
hence
$$f_{T^\bbc}(z)=f(t_z)=f(\xi_z(z))\leq \log|\rho(z).e| - \log|\rho(z).v|\leq C\log|z|\mi.$$
Now we can reparametrise $\xi_z(z)=\log|z|\mi\xi_{z}(1)$. Hence
$$f(\log|z|\mi\xi_{z}(1))=\log|z|\mi f(\xi_z(1))\leq C\log|z|\mi.$$
Because $\xi\mapsto f(\xi(1))$ is proper in $\xi$, this implies that the $\xi_z$ are uniformly bounded, hence subsequentially converge to some $\xi_{z_j}\to\xi_\infty$ with $|z_j|\to 0$. 

We now claim that this limit can be used to compute the slope, i.e.\ that
\begin{equation}\label{eq:slope1}
\frac{|f(\xi_{z_j}(z_j))-f(\xi_\infty(z_j))|}{\log|z_k|\mi}\to 0.
\end{equation}
To that end, we first pick bases $(e_i)$, $(v_i)$ of $E$ and $V$ adapted to the action of $T^\bbc$, i.e.\ so that $E$ and $V$ admit weight decompositions adapted to such bases, with weights $\lambda_i^E$, $\lambda_i^V$. In this decomposition we write $\rho(z).e_i=\sum_j e_{ij}(z) e_j$ and $\rho(z).v_i=\sum_l v_{il}(z)v_l$, with $e_{ij},v_{il}\in \field=\bbc\llparenthesis z\rrparenthesis$. 

Assume that $E$ and $V$ admit the $e_i$, $v_i$ as an orthonormal basis with respect to the chosen inner products.  Then for any $\xi,\xi'\in N_\bbr$ we obtain

\begin{align*}|f(\xi(z))-f(\xi'(z))|=2\left|\log\frac{\sum_i |z|^{2\langle\lambda^E_i,\xi\rangle}\left(\sum_j |e_{ij}(z)|^2\right)}{\sum_i |z|^{2\langle\lambda^E_i,\xi'\rangle}\left(\sum_i |e_{ij}(z)|^2\right)}-\log\frac{\sum_i |z|^{2\langle\lambda^V_i,\xi\rangle}\left(\sum_l |v_{il}(z)|^2\right)}{\sum_j |z|^{2\langle\lambda^V_i,\xi'\rangle}\left(\sum_l |v_{il}(z)|^2\right)}\right|
\end{align*}
which is bounded by $(C\log|z|)|\xi-\xi'|$, where $C$ is independent of $z,\xi,\xi'$ (in fact depending only on the $\lambda_i^E,\lambda_i^V,\ord_0(e_{ij})$ and $\ord_0(v_{il})$---thus only on $\rho$ and the $T^\bbc$-action). Equation \ref{eq:slope1} then follows from this.

From this we deduce the slope formula. On the one hand, using the slope formula Lemma \ref{lem:slope} we find that
$$\lim_{z\to 0} \frac{\inf_{t\in T^{\C}}(\log\|t.\rho(z).e\| - \log\|t.\rho(z).v\|)}{\log|z|\mi} \leq \|(\rho,[e,v])\|_{T^\bbc}.$$
On the other hand, using \eqref{eq:slope1} we obtain
\begin{align*}\lim_{z\to 0} \frac{\inf_{t\in T^{\C}}(\log\|t.\rho(z).e\| - \log\|t.\rho(z).v\|)}{\log|z|\mi}&=\lim_{k\to \infty}\frac{f(\xi_\infty(z_k))}{\log|z_k|\mi}, \\ &=\|(\rho_{\xi_\infty},[e,v])\|, \\ &\geq \|(\rho,[e,v])\|_{T^\bbc},\end{align*}
concluding the proof.
\end{proof}

Using this, we may prove the following numerical criterion.

\begin{theorem}\label{polystable-HM} Assume the norm functional is proper.  Then pair $[v,w]$ is polystable if and only if it is numerically polystable.

\end{theorem}

\begin{proof} Suppose $[v,w]$ is not polystable, so that by definition $\overline{G^H.[C_u]}\cap \Hilb_{\des} \neq \varnothing$, where $u= [e^{\otimes \deg V}\otimes v^{\otimes k+1}, w^{\otimes k+1}\otimes v]$. Then as in Theorem \ref{thm:HMcriterion},  we may find an arc $\rho$ specialising $[C_u]$ to  $\Hilb_{\des}$, so we obtain a sequence $C_{p_z}$ for $z\to 0$ such that, by definition of the Hilbert scheme and ultimately the flat limit, the limit $[D] \in \Hilb_{\des}$ of  $[C_{u_z}]$ is obtained by taking the Zariski closure of the $C_{u_z}$ across the origin in $\Spec\ring$. The Zariski closure agrees with the closure in the analytic topology, so all limits of all possible sequences of the form $t_z.g_z.u$ lie in $D$ as $z\to 0$, where $g_z$ denotes the elements of $G$ associated to $\rho$. 

We claim that the arc $\rho$  that we have constructed satisfies $$(k+1)\nu(\rho, [v,w]) < \|(\rho,[v])\|_{T^{\C}}.$$ Consider the $\xi$ infimising the reduced norm produced by Proposition \ref{prop:slopeequivfindim}, and let $\rho_{\xi}(z)$ be the associated path in $\pr^N$. Then by polyinstability, the limit set $D$ along all possible twists lies entirely in $\pr\left(0\oplus (W^{\otimes k+1}\otimes V)^{H^{\C}}\right)$, so the limit along $ \rho_{\xi}(z)$ lies in $\pr\left(0\oplus (W^{\otimes k+1}\otimes V)^{H^{\C}}\right)$. It follows from the proof of Theorem \ref{thm:HMcriterion} that the weight, which is $(k+1)\nu(\rho, [v,w])- \|(\rho,[v])\|_{T^{\C}},$ must be strictly negative, a contradiction.

Conversely, if $[v,w]$ is not numerically polystable, then consider a destabilising arc $\rho$ with infimiser of the norm $\xi$. Thus if $t_z'$ is such that it infimises $\log \|t_z.(g_z)^{\otimes \deg V}\|-\log\|t_z.g_z.v\|$ over $T^{\C}$, then since by assumption the arc destabilises, it follows that the limit of the $t_z'.g_z(u)$ has limit lying in $\pr\left(0\oplus (W^{\otimes k+1}\otimes V)^{H^{\C}}\right)$. We claim that this is further true for all possible limits of all possible sequences of $g_z(u)$, which will be enough to conclude. To show that, firstly note that by definition, a sequence of points $t_z.g_z(p)$ converges to a point in $\pr(0\oplus (W^{\otimes k+1}\otimes V))$ if and only if $$(\log|g_z.v|-\log|g_z.w|)-(k+1)\mi(\log |t_z.(g_z)^{\otimes \deg V}|-\log|t_z.g_z.v|)\to -\infty,$$ and we know that for $t_z'$ the infimisers of the norm $$\log |t_z.(g_z)^{\otimes \deg V}|-\log|t_z.g_z.v|$$ over arbitrary $t_z$ that this holds. Since \begin{align*}&(\log|g_z.v|-\log\|g_z.w|)-(k+1)\mi(\log |t_z.(g_z)^{\otimes \deg V}|-\log|t_z.g_z.v|) \\ & \leq  (\log|g_z.v|-\log|g_z.w|)-(k+1)\mi(\log |t\_z.(g_z)^{\otimes \deg V}|-\log|t'_z.g_z.v|),\end{align*} and since the right hand side is of the form $$((k+1)\nu(\rho, [v,w]) - \|(\rho,[v])\|_{T^{\C}})\log|z|\mi,$$ it tends to $-\infty$ as $z\to 0$ since the arc is destabilising, so it follows that the sequence of points $t'_z.g_z(p)$ converges to a point in $\pr\left(0\oplus (W^{\otimes k+1}\otimes V)^{H^{\C}}\right)$ as well. Thus the resulting $D$, which consists of all possible limit points  of the $C_{g_z(u)}$, lies in $\Hilb_{\des}$, and so $[v,w]$ is not polystable.
\end{proof}

We next prove that polystability is further equivalent to analytic polystability.

\begin{corollary}  Assume the norm functional is proper.  A pair $[v,w]$ is polystable if and only if it is analytically polystable. \end{corollary}

\begin{proof}

If $[v,w]$ is analytically polystable, then it is numerically polystable by Theorem \ref{prop:slopeequivfindim} and hence is polystable by Theorem \ref{polystable-HM}.

Conversely, suppose $[v,w]$ is not analytically polystable. Then we find a sequence $g_j, t'_j$, with $t'_j$ achieving the infimimum of the norm, such that $t'_j.g_j(u)$  have limit lying in $ \pr\left(0\oplus (W^{\otimes k+1}\otimes V)^{H^{\C}}\right)$ as $j\to\infty$. It follows from the same argument as Proposition \ref{polystable-HM} that in fact then for any possible sequence $t_j \in T^{\C}$, the $t_j.g_j(u)$ also converges (subsequentially, along any possible subsequence) to a point in $\pr\left(0\oplus (W^{\otimes k+1}\otimes V)^{H^{\C}}\right)$. So  all limit points of $T^{\C}.g_j(u)$  lie in $\pr\left(0\oplus (W^{\otimes k+1}\otimes V)^{H^{\C}}\right)$,  meaning that along this sequence $[C_{g_j(u)}]$ converges to a point in $\Hilb_{\des}$. This implies that $$\overline{G^H.[C_p]}\cap \Hilb_{\des} \neq \varnothing,$$ proving polyinstability and hence the result. \end{proof}

Thus obtain the main result of the section:

\begin{corollary}  Assume the norm functional is proper.  Then the following are equivalent:
\begin{enumerate}[(i)]
\item the pair $[v,w]$ is polystable;
\item  the pair $[v,w]$ is numerically polystable;
\item  the pair $[v,w]$ is analytically  polystable.
\end{enumerate}
\end{corollary}

\section{Applications to the Mabuchi functional}\label{sec:applications}

We next apply the theory of stability of pairs to link K-stability to K\"ahler geometry. We begin by recalling the definition of various energy functionals and Paul's work on the Mabuchi functional, before proving slope formulas for the Mabuchi functional along arcs, which will allow us to prove Theorem \ref{intromainthm}, which is the main result of our paper. 

\subsection{Functionals in K\"ahler geometry}\label{sec:functionals}

Assume $(X,L)$ is a smooth polarised variety, and recall from Section \ref{sect:normsaut} that $H$ is a maximal compact subgroup of the reductive part of $\Aut(X,L)$, with $T\subset H$ a maximal torus. Fix $\omega\in c_1(L)$ a $H$-invariant K\"ahler metric and let $$\H = \left\{\phi\in C^{\infty}(X,\R): \omega+\ddb\phi>0\right\};$$ for $\phi \in \H$ we set $\omega_{\phi} = \omega+\ddb \phi$. We denote by $\H^H$ the space of $H$-invariant K\"ahler potentials. The reference form $\omega$ will be fixed throughout and is left implicit in the notation.

The complex torus $T^{\C}$ acts on $\H$ by requiring that for $t\in T^{\C}$ $$\omega+\ddb t. \phi = t^*(\omega+\ddb \phi);$$ in this way $t. \phi$ is defined only up to the addition of a constant, and when using this action of $T^{\C}$ on $\H$ we will only use the equivalence class of $\phi$ under addition of constants.

\begin{definition}\label{def:functionals} We define the  \emph{Mabuchi functional} of $\phi \in \H$  to be $$M(\phi) = H(\phi) + R(\phi)+\frac{n}{n+1}\mu(X,L)E(\phi),$$ where

\begin{enumerate}[(i)]
    \item $E(\phi) = \frac{1}{(n+1)L^n}\sum_{j=0}^{n+1}\int_X\phi\omega^j\wedge\omega_{\phi}^{n-j};$
 \item  $R(\phi) =\frac{1}{L^n}\int_X\phi\left(\sum_{j=0}^{n-1}\omega_{\phi}^j\wedge\omega^{n-j}\right)\wedge \Ric \omega;$
 \item $H(\phi) =\frac{1}{L^n}\int_X \log\left({\frac{\omega_{\phi}^n}{\omega^n}}\right)\omega_{\phi}^n.$
\end{enumerate}
We further define the \emph{$J$-functional} and the  \emph{reduced $J$-functional} to be 
\begin{align*}J(\phi) = \int_X\phi\omega^n - E(\phi); \qquad J_{T}(\phi) = \inf_{t\in T^{\C}}J(t. \phi).
\end{align*}
\end{definition}

The functionals $J$ and $J_{T}$, much like their non-Archimedean counterparts previously introduced, play a role analogous to norms.

\begin{definition} Let $A\subset \H$ be a subset. We say that the Mabuchi functional is 

\begin{enumerate}[(i)]
    \item \emph{coercive} on $A$ if there exist $\epsilon,\delta>0$ such that for all $\phi \in A$ $$M(\phi)\geq \epsilon J(\phi) - \delta.$$
\item \emph{(reduced) coercive} on $A$ if there exist $\epsilon,\delta>0$ such that for all $\phi \in A$ $$M(\phi)\geq \epsilon J_{T}(\phi) - \delta.$$
\end{enumerate}
 \end{definition}

We will mostly be interested in coercivity on $\H$ or $\H^H$, but also on the space of Fubini--Study metrics under an embedding in projective space. Given an embedding $X\hookrightarrow \pr(H^0(X,rL))$ using global sections of $rL$ for $r>0$, restricting the Fubini--Study metric from $\pr(H^0(X,rL))$ produces a K\"ahler metric $r\mi \omega_{FS} \in c_1(L)$. The space of Fubini--Study metrics is then isomorphic to the quotient $\GL(H^0(X,rL))/\moU(H^0(X,rL))$, and we denote this subset of the space of K\"ahler metrics by $\H_{r}$.

\subsection{Paul's results on the Mabuchi functional}\label{subsec:paul}

We now assume $X$ is embedded into projective space $\pr^N = \pr(H^0(X,rL)^*)$ through the linear system $|rL|$ for some $r$ fixed. Since the embedding is linearly normal, we may associate to $X$ its \emph{Chow point} and its \emph{discriminant point} (or \emph{hyperdiscriminant point}; the terminology is not consistent in the literature), as we next recall.

We begin with the Chow point of $X$. Since $X$ has dimension $n$, a generic linear subspace of $\pr^N$ of dimension $N-n$ meets $X$ in $d= (rL)^n$ points, while a generic linear subspace of dimension $N-n-1$ does \emph{not} intersect $X$. We thus consider the set $$\left\{H \in \Grass\left(N-n-1,\pr^N\right): X \cap H\neq \varnothing\right\} \subset \Grass\left(N-n-1,\pr^N\right);$$ this is a divisor in $\Grass\left(N-n-1,\pr^N\right)$ of degree $d$ \cite[Chapter 3, Proposition B.2.2]{book:gkz}, and thus the vanishing locus of a section $R(X) \in H^0\left(\Grass(N-n-1,\pr^N), \scO(d)\right)$ well-defined up to scaling. We call $R(X)$ the \emph{Chow point} of $X$. A basic property of the Chow point is that $$R(g^*X) = g^*(R(X))$$ for $g\in\GL(N+1)$. 

We next define the discriminant point of $X$, using a similar procedure. Consider the set $$\left\{H \in \Grass\left(N-n,\pr^N\right): |X \cap H|\neq d\right\} \subset \Grass\left(N-n,\pr^N\right);$$ that is, where the intersection does not consist of $d$ points. This is equivalent to asking that there is an $x\in X$ with $$\dim H \cap ET_xX \geq 1,$$ where $ET_xX \subset \pr^N$ is the embedded tangent space. This set is again a divisor in $\Grass\left(N-n,\pr^N\right)$ \cite[Chapter 3, Proposition E.2.11]{book:gkz} and so we obtain a point $\Delta(X) \in H^0\left(\Grass\left(N-n,\pr^N\right), \scO(\deg(\Delta(X)))\right)$ which we call the \emph{discriminant point} of $X$. This point satisfies the same equivariance property $$\Delta(g^*X) = g^*(\Delta(X))$$ for $g\in\GL(N+1)$. Paul considers instead a point defined using the dual variety \cite{paul:hyperdiscriminants}; in the situation we consider here (namely a smooth, linearly normal variety), these are equivalent \cite{kapadia}.

Let us write $$R(X) \in H^0\left(\Grass\left(N-n-1,\pr^N\right), \scO(d)\right)$$ and $$\Delta(X) \in H^0\left(\Grass\left(N-n,\pr^N\right), \scO(d)\right),$$ producing a pair $[R(X),\Delta(X)]$ in the sense of Section \ref{sec:pairs}. We now recall the following fundamental result of Paul (and Tian for the $J$-functional \cite[Lemma 3.2]{tian-cm}), for which we embed $\GL(N+1) \subset \C^{N+1\times N+1}$.

\begin{theorem}[{\cite{paul:hyperdiscriminants}}]\label{thm:paul} There are Hermitian inner products on $V$, $W$ and $\C^{N+1\times N+1}$ such that for $g\in \GL(N+1)$ and $\varphi$ the fixed Fubini--Study metric associated to the embedding, we have \begin{align*}\log|\Delta(g^*X)| - \frac{\deg\Delta(g^*X)}{\deg R(g^*X)}\log |R(g^*X)| &= (L^n)M(g.\phi)+const.; \\ \log|g|-\deg R(g^*X)\mi\log |R(g^*X)| &= J(g.\phi) + const.\end{align*}

\end{theorem}

In \cite{kapadia}, Kapadia further shows that the constants in Theorem \ref{subsec:paul} vanish, if one chooses the Hermitian inner products to be  induced by metrics on Deligne pairings.

\begin{remark}
Variants of Paul's work have been proven by Hashimoto--Keller for the Donaldson functional involved in the theory of Hermite--Einstein metrics \cite{hashimoto-keller}, and by Westrich for the higher K-energies \cite{westrich} (which are variants of the Mabuchi functional related to metrics with harmonic $k$\textsuperscript{th}-Chern form).
\end{remark}

\subsection{Deligne pairings and slope formul\ae }\label{subsect:deligne}
Our next aim is to relate energy functionals in Kähler geometry with their non-Archimedean counterparts associated to models.  This is done by extending  slope formul\ae\, for Deligne pairings to the setting of models. The techniques originate we use originate with the work of  Phong--Ross--Sturm \cite{phongrosssturm} for test configurations (with the general philosophy originating  with Ding--Tian \cite{ding-tian} and Tian \cite{tian:kems, tian-cm} by other techniques). Sharp results for the Mabuchi functional for test configurations were later established by Boucksom--Hisamoto--Jonsson \cite{bhj:asymptotics}, as well as by Hisamoto \cite{hisamoto:toric} and Li \cite{chili:guniform} for the reduced J-functional. Some of our more general results for models have been partially established by the second author under the additional assumption of global (rather than relative) positivity of the various metrics on the model  \cite[Theorem B(iii)]{reb:3}; we give here a more direct proof that drops the global positivity assumption (which does not hold in our applications) and extend it further to the Mabuchi functional and the reduced $J$-functional.

The energy functionals of Section \ref{sec:functionals}  fall into the general theory of (metrised) Deligne pairings, which we first explain. We consider a projective flat morphism $\pi: Y \to S$ of relative dimension $n$; in applications, $S$ will be $\Spec\ring,\, \Spec\field,\, \C,\, \C^*$ or simply a point. Given $n+1$ line bundles $L_0,\ldots, L_n$, the pushforward $$\pi_*(L_0\cdot \ldots \cdot L_n) \in A^{1}(S)$$ is a cycle of codimension one in $S$, and the Deligne pairing construction (for which we refer to \cite{deligne:detcoh,elkik:1,boueri} for more details) produces a natural line bundle $$\langle L_0,\ldots, L_n\rangle_{Y/S}$$ representing this cycle class. This construction is multilinear, symmetric, functorial, and commutes with base change.

We now specialise to our setting: we fix $n+1$ line bundles $L_0,\ldots,L_n$ on $X$ and models $(\X_i,\L_i)$ for each $i=0,\ldots, n$; we do not impose any positivity condition on $X$. By functoriality, we may pass to a resolution of indeterminacy (compatible with the identifications $(\X_{i,\field},\L_{i,\field})\cong (X_{\field},L_{\field})$) and assume that the varieties $\X_i$ are equal to some fixed $\X$. We often extend $\pi: (\X,\L_i)\to\Spec\ring$ to $\pi: (\X_{\C},\L_{i,\C})\to\Spec\C$; the resulting Deligne pairings will be canonically isomorphic on $\Spec\ring$ again by functoriality, following the process of Example \ref{ex:different-models}. The properties of Deligne pairings mentioned above then in particular yield the following:
\begin{enumerate}[(i)]
    \item a canonical isomorphism $\langle \L_0,\dots,\L_n\rangle_{\X_{\field}/\field}\cong \langle L_0,\dots,L_n\rangle_{X_{\field}/\field}$ as line bundles over $\Spec\field$ ;
    \item an identification of $(\Spec \ring,\langle \cL_0,\dots,\cL_n\rangle_{\cX/\ring})$ as a model of its restriction $(\Spec\field,\langle L_0,\dots,L_n\rangle_{X_{\field}/\field})$.

\end{enumerate}

The Deligne pairing construction may be refined at the level of metrics, as established by Elkik \cite{elkik:2}. The theory again applies in generality, but we limit our exposition to the setting of our applications. In order to invoke analytic geometry, we use an extension of the models and metrics to $\pi: (\X_{\C},\L_{i,\C}) \to \C$, though one could equally work, for example, over the unit disc $\Delta$. We may further assume that $\X_{\C}$ admits a morphism  $\sigma: \X \to X\times\C$ (namely that it is dominant in the sense of Definition \ref{def:dominant}) extending the identification $(\X_{\C^*},\L_{i,\C^*}) \cong (X\times\C^*, L_i)$, by another application of resolution of indeterminacy. We then endow each $\L_i$ with a Hermitian metric $h_i$ with curvature $\omega_i \in c_1(\L_i)$. Then Elkik constructs a Hermitian metric $\langle h_0,\ldots h_n\rangle_{\X_{\C}/\C}$ (which we call a Deligne metric) on the Deligne pairing $\langle \L_0,\ldots \L_n\rangle_{\X_{\C}/\C}$, which is continuous by Moriwaki \cite{moriwaki:deligne} and is suitably symmetric, functorial and satisfies the appropriate base-change property. 

Given a second metric $h_0e^{-\phi_0}$ on $\L_0$, the quotient of the metrics $\langle h_0e^{-\phi_0},h_1,\ldots,h_n\rangle$ and $\langle h_0,h_1,\ldots,h_n\rangle$ can be identified with (the exponential of minus) a function on $\C$, and the Deligne metric satisfies the  ``change of metric'' formula that this function takes value at $z\in \C$ $$\int_{X_z}\phi_0\omega_{1,z}\wedge\ldots\wedge\omega_{n,z}.$$

We endow each $L_i$ with a reference Hermitian metric $h_i^{\mathrm{ref}}$ with curvature $\omega_i \in c_1(L_i)$, producing a Deligne metric on $\langle L_0,\ldots,L_n\rangle_{X/\{pt\}}$ of $X$ over a point (namely a Hermitian metric on the one-dimensional vector space $\langle L_0,\ldots,L_n\rangle_{X/\{pt\}}$), and similarly on  $\langle L_{0,\bbc},\ldots,L_{n,\bbc}\rangle_{X_{\C}/\C}$. Through the identification $(\X_{\C^*},\L_{i,\C^*}) \cong (X\times\C^*, L_i)$, we thus have two Hermitian metrics  on each line bundle $\L_{i,z}$ on $\X_z$, and their quotient is (the exponential of minus) a function which we write as $\phi_{i,z} \in C^{\infty}(\X_z,\R)$. By the Deligne pairing machinery we thus further obtain two metrics on $\langle \L_0,\dots,\L_n\rangle_{\X_{\field}/\field}=\langle L_0,\dots,L_n\rangle_{X_{\field}/\field}$ whose quotient is  (the exponential of minus) a function $f: \C\to\R$ which we write $$f(z)=\langle \phi_{0,z},\ldots, \phi_{n,z}\rangle:=-\log\langle h_{0,z},\dots,h_{n,z}\rangle_{X/\{pt\}}+\log\langle h_0^{\mathrm{ref}},\dots,h_n^{\mathrm{ref}}\rangle_{X/\{pt\}},$$ where we use the identification of $\X_z$ with $X$.

We consider the value of this function as $z\to 0$. The following result will only involve the germ of the forms $\omega_i$ near $\X_0$, so can meaningfully be considered as the behaviour of $\omega_i$ restricted to $\X \to \Spec\ring$ itself. The following should be compared to Lemma \ref{lem:slope}.

\begin{theorem}\label{thm:slope} As $z\to 0$ the  function $f(z)$ satisfies $$f(z) = (\L_0\cdot\ldots\cdot \L_n)\log|z|^{-1} + O(1).$$
\end{theorem}

\begin{proof}

    We extend the model to a model over $\pr^1$, using the same notation. By assumption, the model $\X$ is dominant, so we have  a morphism $\sigma:\X \to X\times \pr^1$ to the trivial model (as we may pass to a resolution of indeterminacy if not), along with an identification $\X_{\pr^1 -\{0\}} \cong X_{\pr^1 -\{0\}}$ compatible with the various line bundles. We thus obtain an identification $$\langle \L_0,\ldots,\L_n\rangle_{\X_{\pr^1 -\{0\}}/\pr^1 -\{0\}} \cong \langle L_0,\ldots,L_n\rangle_{X\times \pr^1 -\{0\}/\pr^1 -\{0\}},$$ where the latter is the trivial line bundle. 
    
    We choose a trivialising section $s$ of the Deligne pairing $\langle L_0,\ldots,L_n\rangle_{X\times \pr^1/\pr^1}$, chosen to be a constant section of the trivial bundle, so that the pullback $\sigma^*s$ is a trivialising section of $\langle \L_0,\ldots,\L_n\rangle_{\X_{\pr^1 -\{0\}}/\pr^1 -\{0\}} $ away from $\{0\}\in\pr^1$. By the argument of Lemma \ref{lem:slope}, the difference of the degrees of $\langle \L_0,\ldots,\L_n\rangle_{\X_{\pr^1}/\pr^1}$ and $\langle L_0,\ldots,L_n\rangle_{X\times \pr^1/\pr^1}$ is given by the smallest integer $m$ such that $z^m\sigma^*s$ is regular at the origin. Since the second of these Deligne pairings is trivial, this degree vanishes and so the integer $m$ equals the degree of $\langle \L_0,\ldots,\L_n\rangle_{\X_{\pr^1}/\pr^1}$, which is given in turn by the intersection number $\L_0\cdot\ldots\cdot \L_n$, by the general theory of Deligne pairings.
    
    The function $f: \C^*\to\R$, which computed from the quotient of the metrics, can be computed through the section $\sigma^*s$ as 
       \begin{align*}f(z)&=-\log|s(z)|_{\langle h_{0,z},\dots,h_{n,z}\rangle_{X/\{pt\}}}+\log|s(z)|_{\langle h_0^{\mathrm{ref}},\dots,h_n^{\mathrm{ref}}\rangle_{X/\{pt\}}},\\
    &= \log|z^{-m}|-\log|z^{-m}s(z)|_{\langle h_{0,z},\dots,h_{n,z}\rangle_{X/\{pt\}}}+\log|s(z)|_{\langle h_0^{\mathrm{ref}},\dots,h_n^{\mathrm{ref}}\rangle_{X/\{pt\}}}.
    \end{align*}

    The term $\log|s(z)|_{\langle h_0^{\mathrm{ref}},\dots,h_n^{\mathrm{ref}}\rangle_{X/\{pt\}}}$ is constant since $s(z)$ is the constant section of the trivial line bundle $\langle L_0,\ldots,L_n\rangle_{X\times \pr^1/\pr^1}$, while since $z^{-m}\sigma^*s(z)$ is a regular section which does not vanish at the origin, the term $\log|z^{-m}s(z)|_{\langle h_{0,z},\dots,h_{n,z}\rangle_{X/\{pt\}}}$ is bounded near zero. Thus \begin{align*}f(z) &= -m\log|z| + O(1), \\ &= (\L_0\cdot\ldots\cdot \L_n) \log|z|\mi + O(1) \end{align*} as required. \end{proof}

\begin{remark} The usage of Deligne pairings and metrics is not essential here, and it is also possible to argue more directly, following \cite{dervanross, zak:kahler}. \end{remark}

As mentioned at the beginning, the functionals of Section \ref{sec:functionals} can be interpreted through Deligne pairings; for example, the functional $E(\phi)$ is the Deligne metric associated to the Deligne pairing $\langle \L,\ldots,\L\rangle$. We denote these functionals, which now depend on $z\in \C$, by $E(z), R(z)$ and $J(z)$ respectively; these are associated to an arbitrary fixed relatively K\"ahler metric $\omega\in c_1(\L)$ and K\"ahler metric $\omega^{\mathrm{ref}}\in c_1(L)$ induced by metrics $h$ and $h^{\mathrm{ref}}$ on $\L$ and $L$ respectively). The following thus is an immediate consequence of Theorem \ref{thm:slope}.

\begin{corollary}\label{coro:slopes}
Let $(\cX,\cL)$ be a model of $(X,L)$. Then
\begin{enumerate}[(i)]
	\item $E(z)=\mE^\na(\cX,\cL)\log|z|\mi + O(1)$;	
	\item $R(z)=\mR^\na(\cX,\cL)\log|z|\mi + O(1)$;
	\item $J(z)=\|(\cX,\cL)\|\log|z|\mi + O(1)$.
\end{enumerate}
\end{corollary}

We also a slope formula for the reduced $J$-functional, extending work of Hisamoto \cite{hisamoto:toric} and Li \cite{chili:guniform} in the case of test configurations.

\begin{corollary}\label{coro:slopejred}
In the setting of Corollary \ref{coro:slopes}, we have that
$$J_{T}(z)=\|(\cX,\cL)\|_{T^{\C}}\log|z|\mi + O(1)$$
\end{corollary}
\begin{proof}
By Theorem \ref{thm:paul}, the $J$-functional is exactly of the form of the norm functional of Definition \ref{def:anal-poly}. By  \cite[Lemma 2.15]{chili:guniform}, the norm functional is proper in the sense of Definition \ref{def:proper}, so Proposition \ref{prop:slopeequivfindim} implies that the slope of $J_T$ is the infimimum of the slope of $J$ along all twists by one-parameter subgroups. The result then follows since this latter slope is $\|(\cX,\cL)\|$ by Corollary \ref{coro:slopes}.\end{proof}

We finally prove a formula the asymptotics of the entropy (and hence the Mabuchi functional), adapting work of Boucksom--Hisamoto--Jonsson for test configurations \cite[Theorem 3.6]{bhj:asymptotics}. The argument is in essence unchanged for models, and so we simply sketch their argument.

\begin{proposition}\label{prop:slopemabuchi}
	In the setting of Corollary \ref{coro:slopes}, we have that
\begin{enumerate}[(i)]
	\item $H(z)=\mH^\na(\cX,\cL)\log|z| + O(\log(-\log|z|))$;	
	\item $M(z)=\mM^\na(\cX,\cL)\log|z| + O(\log(-\log|z|))$. 
\end{enumerate}
\end{proposition}
\begin{proof}
    By Corollary \ref{coro:slopes}, the analogous holds for the $E$ and $R$ terms in $M$, so it suffices to prove (i).
    
    By passing to a (log) resolution of indeterminacy (which leaves $\mM\NA$ unchanged), we may assume $\X$ is smooth and $\X_0$ has simple normal crossings support.    In particular $K^\mathrm{log}_{\cX/\C}$ is a line bundle, so we may fix a smooth metric $h^{\mathrm{can}}$ on it, along with a reference metric $h^{\mathrm{can}, \mathrm{ref}}$ on $K_X$.  Theorem \ref{thm:slope} implies $$\langle h^{\mathrm{can}}_z,h_z,\ldots,h_z\rangle_{X/\{pt\}}-\langle h^{\mathrm{can}, \mathrm{ref}}, h^{ \mathrm{ref}}\dots, h^{ \mathrm{ref}}\rangle_{X/\{pt\}}=H^\na(\cX,\cL)\log|z|\mi+O(1),$$ so it suffices to show that
    $$\left\langle (\omega_z^n)^*,h_z,\ldots,h_z\right\rangle_{X/\{pt\}}-\langle h^{\mathrm{can}}_z,h_z,\ldots,h_z\rangle_{X/\{pt\}}=O(\log(-\log |z|)),$$ where $\omega_z^n$ is the metric on $-K_X$ induced by the curvature $\omega_z$ of $h_z$ and $(\omega_z^n)^*$ is the dual metric.

    One estimate follows from writing the left-hand side as (viewing $(h^{\mathrm{can}, \mathrm{ref}})^*$ as inducing a volume form)
    $$V\mi \int_X \log\frac{\omega_z^n}{(h^{\mathrm{can}, \mathrm{ref}})^*}\omega_z^n\geq -\log \int_X (h^{\mathrm{can}, \mathrm{ref}})^*$$
    which is $O(\log(-\log|z|^d))$ for some $0\leq d\leq n$ by the main result of \cite{bj:trop}. For the reverse estimate, it suffices to prove that the function
    $ \frac{\omega_z^n}{(h^{\mathrm{can}, \mathrm{ref}})^*}$
    is uniformly bounded above. For the reverse estimate, write the (simple normal crossings) central fibre of $\X$ as $\cX_0=\sum a_i E_i$ for $a_i \in \Z_{>0}$. We may then introduce local coordinates $(w_i)_i$ near any point in the central fibre, such that $\Pi_i w_i^{a_i}=\varepsilon z$. The result then amounts to establishing the inequality of $(n,n)$-forms $$i^n dw_0\wedge d\bar w_0\wedge\dots\wedge\widehat{dw_i\wedge d\overline{w}_i}\wedge \dots\wedge dw_n\wedge d\bar w_n|_{\cX_z}\leq C(h^{\mathrm{can}}_z)^*$$
    for all $i\leq p$ ($p$ being maximal index of the $E_i$) and some positive constant $C$ independent of $i$ and $z$. This then follows from a computation in logarithmic polar coordinates exactly as in \cite[Lemma 3.10]{bhj:asymptotics}.
\end{proof}

\subsection{Proof of the main result} We are now in a position to prove our main result. Recall we have a smooth polarised variety $(X,L)$ with $T^{\C}$ a maximal torus of automorphisms in the reductive part of $\Aut(X,L)$, and denote $\H^H_r$ the space of $H$-invariant Fubini--Study metrics induced by embeddings in $\pr(H^0(X,rL))$. The following proves Theorem \ref{intromainthm}.

\begin{theorem}\label{thm:polyeq}
Fix $r>0$ such that $rL$ is very ample. The following are equivalent:
\begin{enumerate}[(i)]
	\item there exist $C,\varepsilon>0$ such that $M(\phi)\geq \varepsilon J_{T}(\phi)-\delta_r$ for all $\phi\in\cH^H_r$;
	\item there exists $\varepsilon>0$ such that $\mM^\na(\cX,\cL)\geq \varepsilon \|(\cX,\cL)\|_{T^{\C}}$ for all $H^{\C}$-equivariant models $(\cX,\cL)$ of $(X,L)$ (of exponent $r$);
	\item the pair $[v,w]$ associated to $X$ is polystable.
\end{enumerate}
\end{theorem}
\begin{proof}
It follows from Theorem \ref{thm:paul}, due to Paul, that the Mabuchi functional can be expressed within the theory of stability of pairs, where the pair $[R(X),\Delta(X)]$ consists of the Chow point and the discriminant point. Similarly, by Corollary \ref{coro:slopejred}, the reduced $J$-functional is the norm functional involved in the setup. By Theorem \ref{polystable-HM}, it follows that $(i)$ is equivalent to polystability of the pair $[R(X),\Delta(X)]$, which in turn is equivalent (by the same result) to numerical polystability of the pair. 

The arcs involved in numerical polystability then correspond to models of exponent $r$, by Proposition \ref{prop:arcsvmodels}, and the various numerical invariants can be computed by restricting the Fubini--Study metric from the overlying projective space. A combination of Proposition \ref{lem:slope}, Proposition \ref{prop:slopemabuchi} and  Corollary \ref{coro:slopejred}, imply that the weight relevant to numerical polystability is $\mM^\na(\cX,\cL)- \varepsilon \|(\cX,\cL)\|_{T^{\C}}$. Thus this latter condition, over all models of exponent $r$, is equivalent to coercivity on $\phi\in\cH^H_r$. \end{proof}

Corollary \ref{coro:intro1} is a consequence.

\begin{corollary}\label{body:cc-cor}
 If the class $c_1(L)$ admits a cscK metric, then $(X,L)$ is uniformly K-polystable with respect to arcs.
\end{corollary}

\begin{proof}
By Berman--Darvas--Lu \cite{BDL:aens}, the existence of a cscK metric implies that the Mabuchi functional is coercive: there exists $\epsilon, \delta>0$ such that for all $\phi \in \H^H$ $$M(\phi) \geq \epsilon J_T(\phi) - \delta.$$ In particular, the Mabuchi functional is coercive on each space $\H^H_r$. We may view an  arbitrary $H^{\C}$-equivariant model as induced by an arc in $\GL(H^0(X,rL))^{H^{\C}}$ for some $r>0$ by Proposition \ref{prop:arcsvmodels},  so by Theorem \ref{thm:polyeq} we have $$\mM^\na(\cX,\cL)\geq \varepsilon \|(\cX,\cL)\|_{T^\C}.$$ This condition over all exponents $r>0$ is equivalent to uniform K-polystability with respect to arcs (which by definition is measured through the Donaldson--Futaki invariant) by Remark \ref{DFMNA},  proving the result. 
\end{proof}

\section{Structure of the automorphism group}
\subsection{The Donaldson--Futaki invariant in the singular case}
The aim of this section is to prove that uniform K-polystability with respect to arcs of a polarised variety restricts the structure of its automorphism group. To treat the case of singular varieties, we will require another perspective on the Donaldson--Futaki invariant of an arc, closer to the original approach of Donaldson \cite{donaldson:bstab}.

Let $(X,L)$ be a polarised variety, not necessarily irreducible. Fix a model $\pi: (\X,\L)\to\pr^1$ for $(X,L)$, which we have considered as a family over $\pr^1$ by the gluing argument of Example \ref{ex:different-models}. We assume $\L$ is a line bundle (rather than a $\Q$-line bundle) by scaling if necessary; the argument adapts in the obvious manner in general.  

By flatness of $\pi: \X\to\pr^1$ we obtain for $r\gg 0$ a sequence of vector bundles $\pi_*(r\L)$ over $\pr^1$. The Knudsen--Mumford theorem \cite{knudsenmumford} produces an expansion of their determinants $$\det \pi_*(r\L) = \lambda_0^{\otimes {r \choose n+1}} \otimes  \lambda_1^{\otimes {r \choose n}}\otimes\ldots,$$ for line bundles $\lambda_0,\lambda_1$. Write in addition the Hilbert polynomial of $(X,L)$ as $$h^0(X,rL) = a_0r^n+a_1r^{n-1}+\ldots,$$ for $a_0,a_1 \in \Q$.

\begin{lemma}\label{KM-use}
If $X$ is normal, the Donaldson--Futaki invariant satisfies  $$\DF(\X,\L) =\frac{2\deg\lambda_0 a_1 - (n+1)\deg \lambda_1 a_0}{4a_0^2(n+1)!}.$$
\end{lemma}

\begin{proof} This is a statement about the asymptotics of the degree of the vector bundle $\pi_*(r\L)$ over $\pr^1$, and hence follows from Grothendieck--Riemann--Roch. \end{proof}

We take this as our definition of the Donaldson--Futaki invariant when $X$ is not normal, and use this interpretation throughout the present section.  We use this interpretation to prove the following.

\begin{lemma}\label{inversion-lemma}
Let $\rho$ be an arc in $\Aut(X,L)$. Then $\DF(\rho^{-1}) = -\DF(\rho),$ where $\DF(\rho)$ denotes the Donaldson--Futaki invariant of the associated model.
\end{lemma}

\begin{proof}

Embed $X$ into projective space $\pr(H^0(X,rL))$ for $r\gg 0$, so that $\rho$ is also an arc in $\GL(H^0(X,rL))$, and denote by $[v]$ the Hilbert point of $X$, namely the associated point in the Hilbert scheme $\Hilb$, by the Knudsen--Mumford theorem we obtain a sequence of line bundles $\lambda_j$ on $\Hilb$, defined through the relatively ample line bundle $\L$ on the universal family over $\Hilb$. 

Viewing $[v]$ as inducing a point $v$ in projective space through any very ample line bundle over $\Hilb$ with a $\GL(H^0(X,rL))$-linearisation, since $\rho(z)$ fixes $[v]$ (as $\Aut(X,L)$ is the stabiliser of $[v]$), there exists a Laurent series $c(z)\in \field$ such that for all $z$,
$$\rho(z).v = c(z).v,$$
or equivalently $\rho.(v)_\field= c\cdot (v)_\field$. The weight associated to this ample line bundle thus satisfies by definition $\nu(\rho,[v])=\ord_0(c)$, the order of pole of the Laurent series $c$. Then $$v = \rho^{-1}(z).\rho(z).v = \rho^{-1}(z)c(z).v = c(z) \rho^{-1}(z).v,$$ so $\rho^{-1}(z).v = c^{-1}(z).v$. This means that the weight of $\rho$ on $v$ is minus the weight of $\rho^{-1}$ on $v$, with respect to this very ample line bundle. 

We may write the line bundles $\lambda_0, \lambda_1$ as differences of ample line bundles, so that the weight satisfies the same property. It follows that $\DF(\rho^{-1}) = -\DF(\rho)$ by Lemma \ref{KM-use}.\end{proof}

This result is classical for one-parameter subgroups $\lambda\hookrightarrow \Aut(X,L)$. Note that this result also implies the analogous claim for $\mM\NA(\X,\L)$, since for a product arc the central fibre $\X_0 = X$ is reduced.

\subsection{Automorphisms and arcs}\label{sec:aut-arcs}

We aim to understand the structure of the identity component $\Aut_0(X,L)$ of the automorphism group $\Aut(X,L)$ of $(X,L)$ in relation to K-stability. We begin with basic facts.

If $\lambda: \C^*\hookrightarrow \Aut(X,L)$ is a one-parameter subgroup, we obtain a product test configuration by considering $(\X,\L) = (X\times \C^*,L)$ with $\C^*$-action given by $(x,z) \to (\lambda(t).x,tz)$. This induces an arc by restricting to $\Spec \ring$, the formal neighbourhood of the origin in $\C$, where the identification $(\X_{\field},\L_{\field}) \cong (X_{\field},L_{\field})$ is induced by the inverse of the $\C^*$-action, as discussed in Section \ref{sect:normsaut}.

Suppose instead that $\rho_+: (\C,+)\hookrightarrow \Aut_0(X,L)$ is a subgroup isomorphic to the additive group $(\C,+)$, and consider the action of the latter on $\pr^1$ given for $s\in \C$ in chosen coordinates by $$[z_0:z_1] \to s.z:=[z_0+sz_1:z_1];$$ when $z_0\neq 0$ we have $$[z_0+sz_1:z_1] = \left[1:\frac{z_1}{z_0+sz_1}\right].$$ This action has unique fixed point at infinity, $[1:0] \in \pr^1$. The product action of $(\C,+)$ on $(X\times\pr^1,L)$ given by $(x,z) \to (\rho_+(s).x,s.z)$ has a fixed fibre over $[1:0]\in\pr^1$, and taking a formal neighbourhood of $[1:0]$ produces a model $(\X,\L)$ of $(X,L)$. Explicitly, the identification $(\X_{\field},\L_{\field}) \cong (X_{\field},L_{\field})$ is given by restricting the identification (on the locus of $\pr^1$ where $z_0\neq 0$) of $X\times\C\cong X\times \C$ defined by $$(x,z) \to \left(\rho_+(s^{-1}).x, \frac{z}{1+s^{-1}z}\right);$$ the inverse is so that the interesting behaviour is as $s\to 0$ rather than $s\to \infty$. This is simply a more explicit version of the model induced by $\rho_+$, viewed as an arc in $\Aut(X,L) \subset \GL(H^0(X,rL))$.

\subsection{Uniform K-stability implies finite automorphisms}

Consider now a polarised variety $(X,L)$, which we allow to be singular.

\begin{theorem}\label{thm:nounipotent}
If $(X,L)$  is uniformly K-stable with respect to arcs, then $\Aut(X,L)$ is finite.
\end{theorem}

We have proven that uniform K-stability with respect to arcs implies stability of the pair in any fixed projective space, and Paul has proven that a stable pair has finite automorphism group  \cite[Proposition 4.10]{paul-stablepairs-13}. Thus from the results of Section \ref{sec:applications} and Paul's work, we obtain the above result as a corollary. We give here a more direct proof that is amenable to further generalisation.

\begin{proof}

Let $[v]$ be the Chow point associated to $X$. Recall that the norm of an arc $\rho$ is defined to be
$$\|(\rho,[v])\|=\nu\left(\rho,\left[v,e^{\otimes \deg V}\right]\right)$$
where $V$ is the vector space of definition of the Chow point $V$, and $e$ is the identity in $\GL(H^0(X,rL))$.

Note first that Lemma \ref{inversion-lemma} together with the stability hypothesis implies that that any arc $\rho$ in  $\Aut(X,L)$ must satisfy $\|\rho\|=0$, where the latter denotes the norm of the induced model, by definition of uniform K-stability. 

To show that  $\Aut(X,L)$ is finite, it suffices to show that it contains no subgroups of the form $\C^*$ or $(\C,+)$. Since a product test configuration has positive norm \cite[Theorem 1.3]{uniform-twisted}\cite[Corollary B]{bhj:duistermaat}, the former is impossible. 

Thus suppose $\Aut(X,L)$ contains a subgroup of the form  $(\C,+)$, with variable $s$, producing an arc denoted $\rho_+$, following the process of Section \ref{sec:aut-arcs}.  Setting as before $\rho_+(z).v=c(z)\cdot v$, the Laurent series $c$ must be identically one as $(\C,+)$ is unipotent, and unipotent matrices have eigenvalues zero or one. One can also see this, following Paul \cite[Section 4.4]{paul-stablepairs-13}, by using that the assignment $z\to c(z)$ is a character of $(\C,+)$, and characters of unipotent groups are trivial. Thus, the weight on the Chow point---equal to the order of vanishing of the Laurent series $c$---is zero. 

The remaining term defining $\|\rho_+\|$ may be computed by fixing a matrix norm, which by Lemma \ref{lem:slope} allows us to reduce the computation to that of the asymptotics as $s\to 0$ of   $$
     \log\left|\left( {\begin{array}{cc}
   1 & s^{-1} \\
   0 & 1 \\
  \end{array} } \right)\right|
 + O(1) = \log |s|^{-1} + O(1),$$ which is enough to conclude that $\|\rho_+\|>0$, giving a contradiction. \end{proof}

\subsection{Reductivity}
We turn to the general case. We fix a maximal torus $T^{\C}$ of $\Aut(X,L)$ and consider the centraliser of $T^{\C}$ in the automorphism group $\Aut_0(X,L)$, namely the elements commuting with the maximal torus. Note that this is also the group agrees with the commutator of $H^{\C}$, since $T^{\C}\subset H^{\C}$ is a maximal torus. We denote this group by $\Aut_0(X,L)^{T^{\C}}$; this is known as the Cartan subgroup of $\Aut(X,L)$ associated with $T^{\C}$. We begin with the following observation.

\begin{lemma} The group $\Aut_0(X,L)^{T^{\C}}$ satisfies $\Aut_0(X,L)^{T^{\C}} \cong  T^{\C} \oplus U $.
\end{lemma}

\begin{proof}
By general theory, the Lie group $\Aut_0(X,L)^{T^{\C}}$ may be written as the semidirect product of a unipotent group $U$ with a reductive group $R$, and since $R$ commutes with $T^{\C}$ it must lie in a maximal torus of $\Aut_0(X,L)^{T^{\C}}$. Since $T^{\C}$ is a maximal torus of $\Aut(X,L)$, it follows that $R=T^{\C}$. Thus $\Aut_0(X,L)^{T^{\C}}=U\ltimes T^{\C}$, and our final claim is that this is actually a trivial semidirect product. 

This semidirect product is equivalent to the data of a group homomorphism $\phi: T^{\C} \to \Aut(U)$ defined by $\phi_t(u) = tut^{-1}$. Since $T^{\C}$ commutes with $U$, it follows that $\phi$ is trivial and thus $\Aut_0(X,L)^{T^{\C}} \cong U \oplus T^{\C}$.\end{proof}

We may now prove our desired reductivity statement.

\begin{theorem}
Suppose $(X,L)$  is uniformly K-polystable with respect to arcs. Then $\Aut_0(X,L)^{T^{\C}}$ is reductive, and hence equals a maximal torus.
\end{theorem}

\begin{proof}
By the same argument as in the proof of Theorem \ref{thm:nounipotent}, as  $(X,L)$ is in particular K-semistable with respect to arcs, it follows that $$\DF(\rho) = -\DF(\rho^{-1})$$ for any arc $\rho$ in $\Aut_0(X,L)^{T^{\C}} \cong T^{\C} \oplus U$ and so any such $\rho$ must satisfy $\|\rho_+\|_{T^{\C}}=0$ (meaning the norm of the induced model). Supposing that there exists a subgroup of $U\subset U\oplus T^\bbc$ isomorphic to $(\bbc,+)$, and denoting by $\rho_+$ the induced arc, it suffices to show that $\|\rho_+\|_{T^{\C}}>0$  to show that such a subgroup cannot exist, and therefore that $U$ is trivial.

As the infimiser defining $\|\rho_+\|_{T^{\C}}$ may not be rational, we first extend the norm from one-parameter subgroups $N_{\Z}$ to general elements of $N_{\R}$. We do so essentially by continuity, and it is most transparent to argue analytically. Fixing a $T$-invariant K\"ahler metric $\omega\in c_1(L)$, we may associate to an arbitrary element $\xi  \in N_{\R}$ a Hamiltonian function $h_{\xi}$. In this way, for $\xi\in N_{\Z}$ we have $$\|\xi\| = \sup_X h_{\xi} - \frac{1}{L^n}\int_X h_{\xi}\omega^n,$$ which is a standard result when $X$ is smooth and is easily seen to also hold when $X$ is singular (for example by resolving singularities and arguing by continuity), and importantly for this extension it remains true that the norm of $\xi \in N_{\R}$ vanishes if and only if $\xi=0$.

Let $\xi$ be the element of $N_{\R}$ achieving the infimum defining $\|\rho_+\|_{T^{\C}}$. If $\xi=0$, we are done by the argument of Theorem \ref{thm:nounipotent}, so assume that $\xi\neq 0$. We claim that 
\begin{equation}\label{eq:lower-bound}
\|\rho_+\|_{T^{\C}} \geq \|\xi\|,
\end{equation} which will imply the result since $\|\xi\| >0$ by assumption. 

To prove the claim, we begin by showing that for any $\lambda$ a one-parameter subgroup in $T^{\C}$ that we have the bound $$\|\rho_+\circ \lambda\| \geq \|\lambda\|.$$ As in the proof of Theorem \ref{thm:nounipotent}, we have $\rho_+(t).v = v$, thus the weights of the Chow point $[v]$ under $\rho_+$ and $\rho_+\circ \lambda$ are equal. It then suffices to compare the matrix part of $\|\rho_+\|$ and $\|\lambda\circ\rho_+\|$. The computation of the latter again reduces to understanding the asymptotics of the log-norm of a block-matrix, namely $$\log\left|
  {  \left(  \begin{array}{cc}
   \lambda(t) & 0 \\
   0 & \rho_+(t) \\
  \end{array}\right) }\right|.$$ The associated weight is then the maximum of the weights of $\lambda(t)$ and $\rho(t)$, giving the lower bound of Equation \eqref{eq:lower-bound} for such $\lambda$. Rhe continuity statement of Lemma \ref{lem:continuity} and continuity of the norm on $N_{\R}$ imply then that for $\xi$ the infimiser of $\|\rho_+\|_{T^{\C}}$ that we have $$\|\rho_+\|_{T^{\C}} = \inf_{\lambda \in N_{\Z}} \|\rho_+\circ\lambda\| \geq \|\xi\|,$$ which proves Equation \ref{eq:lower-bound} and hence the result. \end{proof}

\section{The Yau--Tian--Donaldson conjecture and the partial $C^0$-estimate}

The purpose of this section is to prove Corollary \ref{coro:intro2} from the introduction, and to reduce the general Yau--Tian--Donaldson conjecture to the partial $C^0$-estimate. As explained in the introduction, we build on Tian's programme, following ideas of Tian along the lines of \cite[Section 5]{sze:partial} and \cite[Section 6]{bhj:asymptotics} (who explain the results in the present section for Fano manifolds with discrete automorphism group).

We fix $(X,L)$ a smooth polarised variety, along with reference K\"ahler metrics $\omega, \alpha \in c_1(L)$ assumed to be $H$-invariant (where as usual $H \subset \Aut_0(X,L)$ is a maximal compact subgroup and $T\subset H$ is a maximal torus). We consider the twisted continuity method introduced by Chen \cite{chen-continuity}, namely a collection $\phi_s \in \H^H$ satisfying $$S(\omega_s) - s\Lambda_{\omega_s}\alpha = \hat S_s,$$ where $\hat S_s$ is the appropriate topological constant and $\omega_s = \omega+\ddb \phi_s$. We note that solutions are unique when they exist for $s>0$, and are unique up to biholomorphisms for $t=0$ \cite{berber:kenergy}. Existence for large $t$ is due to Hashimoto and Zeng \cite{hashimoto, zeng}, while openness of the set of $s$ for which there is a solution is due to Chen \cite{chen-continuity}. In the special case of Fano manifolds with $L=-K_X$ the anticanonical class, this is a reparametrisation of the usual continuity method $$\Ric(\omega_s) = s\omega_s + (1-s)\alpha.$$

To an arbitrary $\phi \in \H^H$ inducing a K\"ahler metric $\omega_{\phi}$, for each $r>0$ we obtain a Hermitian inner product on each $H^0(X,rL)$ by taking the $L^2$-inner product; this in turn induces Fubini--Study metric which we denote $\FS_r(\phi)\in\cH^H_{r}$. 

\begin{definition} Consider the set $\cC$ of solutions $\phi_s$ of solutions of the twisted cscK equation, for which a solution exists.  We say that $(X,L)$ satisfies the \textit{partial $C^0$-estimate along the continuity method} in exponent $r$ if there exists $C>0$ such that for all $s$,
$$|\FS_k(\phi_s)-\phi_s|\leq C.$$
\end{definition}

This condition can also be viewed in terms of Bergman kernels. By a fundamental result of Székelyhidi \cite{sze:partial}, if $X$ is Fano and $L=-K_X$, then the partial $C^0$-estimate holds for all exponents $r\gg 0$. 

\begin{theorem}\label{thm:contmethod}
Assume the following:
\begin{enumerate}[(i)]
\item there exists $\varepsilon>0$ such that $M^\na(\cX,\cL)\geq \varepsilon \|(\cX,\cL)\|_{T^{\C}}$ for all $H^{\C}$-equivariant models $(\cX,\cL)$ of $(X,L)$ of exponent $r$;
\item $(X,L)$ satisfies the partial $C^0$-estimate along the continuity method in exponent $r$.
\end{enumerate}
Then there exists a $\delta>0$ such that, for all $\phi_s\in  \cC$ along the continuity method
$$M(\phi_s)\geq \epsilon J_{T}(\phi_s)-\delta.$$
\end{theorem}
\begin{proof}
By Theorem \ref{thm:polyeq}, $(i)$ is equivalent to existence of $\delta'>0$ such that 
$$M(\phi)\geq \varepsilon J_{T}(\phi)-\delta'$$ for all $\phi\in\cH^H_{\omega,r}$. In particular for all $\phi_s\in \cC$,
$$M(\FS_r(\phi_s))\geq \varepsilon J_{T}(\FS_r(\phi_s))-\delta'.$$
By \cite[Lemma 6.5 (ii), Lemma 6.6]{bhj:asymptotics} we have that 
$$M(\FS_r(\phi_s))\leq M(\phi_s)+h^0(X,rL)\sup_X|\FS_k(\phi_s)-\phi_s|\leq M(\phi_t)+h^0(X,rL)C$$
where $C$ is the constant appearing in the partial $C^0$-estimate assumed by (ii). The $J_T$-part is taken care of similarly by Lemma \ref{lem:jtest} below, which gives $$|J_{T}(\FS_r(\phi))-J_{T}(\phi)|\leq 2\sup_X |\FS_r(\phi)-\phi|.$$ As this is bounded through the partial $C^0$-estimate, we obtain a $\delta>0$ such that $$M(\phi_s) \geq \epsilon J_T(\phi_s) - \delta,$$ proving the result.
\end{proof}

We prove the requisite lemma (see also Hisamoto \cite[p.\ 10]{hisamoto:toric})

\begin{lemma}\label{lem:jtest}
For all $\phi,\phi'\in \cH^H$ we have
$$|J_{T}(\phi)-J_{T}(\phi')|\leq 2\sup_X |\phi-\phi'|.$$
\end{lemma}
\begin{proof}
By \cite[Lemma 6.5 (i), Lemma 6.6]{bhj:asymptotics} we have that
\begin{equation}\label{eq:jt1}
|J(\phi)-J(\phi')|\leq 2\sup_X |\phi-\phi'|.
\end{equation}
Furthermore, for all $t\in T^\bbc$,
\begin{equation}\label{eq:jt2}
\sup_X|t.\phi-t.\psi|=\sup_X |\phi-\phi'|
\end{equation}
holds. Indeed, a straightforward calculation shows that for $t\in T^{\C}$ that $$t.\phi-t.\psi=t.(\phi-\psi)=t^*(\phi - \psi),$$ giving \eqref{eq:jt2} since $T^\bbc$ acts by biholomorphisms. Combining \eqref{eq:jt1} and \eqref{eq:jt2}, using that $J$ is nonnegative, we find
$$|J_{T}(\phi)-J_{T}(\phi')|\leq \sup_t |J(t.\phi)-J(t.\phi')|\leq 2\sup_t \sup_X |t.\phi-t.\phi'|= 2\sup_X |\phi-\phi'|$$
as desired.\end{proof}

We immediately obtain the following.

\begin{theorem}
Assume that $(X,L)$ satisfies the partial $C^0$-estimate along the continuity method some $r>0$. Then the following are equivalent:
\begin{enumerate}[(i)]
	\item there exists a cscK metric in $c_1(L)$.
		\item $(X,L)$ is uniformly K-polystable with respect to arcs;

\end{enumerate}
\end{theorem}

\begin{proof}
By Chen--Cheng \cite{chencheng:ii,chencheng:iii}, $(i)$ is equivalent to coercivity of the Mabuchi functional on $\cH^H$, which by Corollary \ref{body:cc-cor} implies uniform K-polystability with respect to arcs. 

Conversely, suppose $(X,L)$ is uniformly K-polystable with respect to arcs, and suppose $(X,L)$ satisfies the partial $C^0$-estimate along the continuity method. Note that the Mabuchi functional is decreasing along the twisted continuity method, since a twisted cscK metric for $s>0$ achieves the unique minimum of the twisted Mabuchi functional \cite{berber:kenergy}, and this functional is monotonic along the continuity method. 

So $M$ is bounded above along the continuity method, and hence Theorem \ref{thm:contmethod} implies---since we assume the partial $C^0$-estimate for some exponent $r$---that we have $$M(\phi_s)\geq \epsilon J_{T}(\phi_s)-\delta$$ and so $J_{T}(\phi_s)$ is also bounded above along the continuity method. By Chen--Cheng \cite{chencheng:ii,chencheng:iii} this implies that $c_1(L)$ admits a cscK metric (which they construct as the limit of the twists $t_s.\phi_s$ computing $J_T(\phi_s)$).
\end{proof}

As Sz\'ekelyhidi has established the partial $C^0$-estimate along the continuity method in the Fano case \cite{sze:partial}, we obtain Corollary \ref{coro:intro2}:

\begin{corollary}
A Fano manifold $X$ admits a K\"ahler--Einstein metric if and only if it is uniformly K-polystable with respect to arcs.
\end{corollary}

\begin{remark} It is not clear whether one should expect this strong form of the partial $C^0$-estimate, namely the uniform bound $|\FS_r(\phi_s)-\phi_s|\leq C$ along the continuity method, to hold in complete generality. We remark that the version used in the introduction, namely an equivalence of coercivity of the Mabuchi functional on the full space of K\"ahler metrics and on each space of Fubini--Study metrics, is equivalent to the arc version of the Yau--Tian--Donaldson conjecture by Theorem \ref{intromainthm}, and in particular is implied by the uniform, test configuration version of the Yau--Tian--Donaldson conjecture. What the present section shows is that the former version of the partial $C^0$-estimate  implies the latter (hence weaker) version, which is expected to hold in general.
\end{remark}

\bibliographystyle{alpha}
\bibliography{bib}

\end{document}